\documentclass[reqno, 11pt, a4paper]{amsart}
\usepackage{amsmath,amsfonts,amsthm,amssymb,amscd}
\usepackage{hyperref}
\usepackage[alphabetic]{amsrefs}
\usepackage{tikz}
\usepackage{graphicx}
\usepackage{xcolor}

\usepackage[hmargin=34mm,vmargin=41.6mm,footskip=5ex]{geometry}

\makeatletter
\def\ps@headings{\ps@empty
  \def\@evenhead{%
    \setTrue{runhead}%
    \normalfont\scriptsize
    \hfil
    \def\thanks{\protect\thanks@warning}%
    \leftmark{}{}\hfil}%
  \def\@oddhead{%
    \setTrue{runhead}%
    \normalfont\scriptsize \hfil
    \def\thanks{\protect\thanks@warning}%
    \rightmark{}{}\hfil}%
  \def\@oddfoot{\normalfont\scriptsize \hfil\thepage\hfil}%
  \let\@evenfoot\@oddfoot
  \let\@mkboth\markboth
}
\makeatother

\numberwithin{equation}{section}	

\newtheorem{lem}{Lemma}[section]

\newtheorem{tw}[lem]{Theorem}
\newtheorem{wtw}{Theorem}
\newtheorem{cor}[lem]{Corollary}
\newtheorem{prop}[lem]{Proposition}

\newcommand {\E}{\mathbb{E}_{\Delta}^2}
\newcommand {\trol}[1]{L(#1)}

\newcommand {\ent}[1]{\lfloor#1\rfloor}

\newcommand {\minset}[1]{\mathrm{Min}(#1)}

\newcommand {\disp}[2]{\mathrm{Disp}_{#1}(#2)}
\newcommand {\link}[2]{\mathrm{Lk}(#1 , #2)}

\newtheorem*{claim0}{Claim}
\newtheorem*{claim1}{Claim 1}
\newtheorem*{claim2}{Claim 2}
\newtheorem*{claim3}{Claim 3}
\newtheorem*{claim4}{Claim 4}

\theoremstyle{definition}
\newtheorem{de}[lem]{Definition}

\theoremstyle{remark}
\newtheorem{rem}[lem]{Remark}

\newtheorem{convent}[lem]{Convention}

\begin{document}

\title[Hyperbolic isometries and boundaries]{Hyperbolic isometries and boundaries of systolic complexes}
\author{Tomasz Prytu\l a}
\date{\today}


\address{School of Mathematics, University of Southampton, Southampton SO17 1BJ, UK}
\email{t.p.prytula@soton.ac.uk}

\begin{abstract}

Given a group $G$ acting geometrically on a systolic complex $X$ and a hyperbolic isometry $h \in G$, we study the associated action of $h$ on the systolic boundary $\partial X$. We show that $h$ has a canonical pair of fixed points on the boundary and that it acts trivially on the boundary if and only if it is virtually central. The key tool that we use to study the action of $h$ on $\partial X$ is the notion of \mbox{a $K$--displacement} set of $h$, which generalises the classical minimal displacement set of~$h$. We also prove that systolic complexes equipped with a geometric action of a group are almost extendable.
\end{abstract}

\maketitle

\section{Introduction}
\label{sec:intro}

A systolic complex is a simply connected simplicial complex whose vertex links satisfy a certain combinatorial condition called $6$--\emph{largeness}. The condition of $6$--largeness serves as an upper bound for the combinatorial curvature, and thus systolic complexes may be seen as combinatorial analogues of metric spaces of nonpositive curvature, the so-called $\mathrm{CAT}(0)$ spaces. 
Systolic complexes were first introduced in \cite{Ch4} under the name of \emph{bridged complexes}, although their $1$--skeleta had appeared much earlier in metric graph theory (see e.g., \cite{SC}). In this article we are interested in systolic complexes that are equipped with a geometric action of a group. Any such group is called a \emph{systolic group}. 
The theory of systolic complexes and groups, as developed in \cite{JS2}, is to a large extent parallel to the theory of $\mathrm{CAT}(0)$ spaces and groups. In particular, over the last fifteen years many of the nonpositive curvature-like properties of systolic complexes have been established (see \cites{Ch4, JS2, E1, E2, E3, OsaPry} and references therein).
On the other hand, a combinatorial approach led to constructions of examples of systolic groups whose behaviour is very different from the classical nonpositively curved groups \cite{JS2}.

An important invariant of a $\mathrm{CAT}(0)$ space $X$ is its \emph{boundary at infinity} $\partial X$. The boundary is a topological space which, as a set, consists of equivalence classes of geodesic rays in $X$, such that asymptotic rays are equivalent. One topologises it in a way that two geodesic rays are `close' if they fellow travel `long time'. Any $G$--action by isometries on $X$ gives rise to a $G$--action by homeomorphisms on $\partial X$. It turns out that many algebraic properties of a group are reflected in topological properties of the boundary and in the action itself. \smallskip

In this article we study this correspondence in the setting of systolic complexes. The boundary for systolic complexes was constructed in \cite{OP}. The construction is similar to the one for $\mathrm{CAT}(0)$ spaces, however it is much more technical. The points of the systolic boundary are also represented by geodesic rays in (the $1$--skeleton of) a systolic complex $X$, but not every geodesic ray in $X$ gives a point in the boundary: in order to ensure good properties of the boundary, a choice of a certain subclass of geodesics was necessary. This is mainly due to the fact that arbitrary geodesics in a systolic complex do not satisfy any form of the Fellow Traveller Property (indeed, two geodesics of length $D$ with the same endpoints may get $\frac{D}{2}$ apart). In \cite{OP} the authors introduce \emph{good geodesics} and \emph{good geodesic rays}, and define the systolic boundary $\partial X$ as a set of equivalence classes of good geodesic rays in $X$. The topology on $\partial X$ is defined analogously as in the $\mathrm{CAT}(0)$ case. Both good geodesics and good geodesic rays are preserved by simplicial automorphisms of $X$, and therefore any simplicial $G$--action on $X$ induces a $G$--action (by homeomorphisms) on $\partial X$. Intuitively, a good geodesic ray is a geodesic ray which, whenever contained in a flat $F$, follows the $\mathrm{CAT}(0)$ geodesic in $F$. In particular good geodesic rays have the desired metric properties, similar to those of geodesic rays in $\mathrm{CAT}(0)$ spaces.\smallskip 

An isometry (i.e.,\ a simplicial automorphism) $h$ of a systolic complex $X$ is \emph{hyperbolic} if it does not fix any simplex of $X$. Note that if $G$ acts geometrically on $X$ then every infinite order element of $G$ is a hyperbolic isometry of $X$. The main point of this article is to study the associated action of $h$ on the systolic boundary $\partial X$. We start by determining when this action is trivial (i.e., when $h$ acts as the identity on $\partial X)$. Denote by $C_G(h)$ the centraliser of $h$ in $G$. The following is the systolic analogue of a result of K. Ruane for $\mathrm{CAT}(0)$ spaces \cite{Kru}.

\begin{wtw}[Theorem~\ref{tw:virtuallycentral}]\label{tw:twa}Let $G$ be a group acting geometrically on a systolic complex $X$, and let $h \in G$ be a hyperbolic isometry. Then $h$ acts trivially on the boundary $\partial X$ if and only if the centraliser $C_G(h)$ has finite index in $G$.
\end{wtw}

The canonical object used to study the action of $h$ on $X$ is the \emph{minimal displacement set} of $h$, which is a subcomplex of $X$ spanned by all the vertices which are moved by $h$ the minimal (combinatorial) distance. This distance is called the \emph{translation length} of $h$ and it is denoted by $\trol{h}$. Due to a `coarse nature' of $\partial X$, in order to study the action of $h$ on $\partial X$ it is convenient to replace the minimal displacement set by its coarse equivalent -- the $K$--\emph{displacement set} of $h$, for some $K\geqslant \trol{h},$ which is a subcomplex of $X$ spanned by all the vertices that are moved by $h$ the distance at most $K$. The $K$--displacement set has all the desired (from our point of view) features of the minimal displacement set, while it has the advantage of being more flexible as one can let $K$ vary.\medskip 

The proof of Theorem~\ref{tw:twa} is based on the interplay between $K$--displacement sets of $h$ (for different values of $K$) and the centraliser $C_G(h)$. In particular, the `if' direction essentially boils down to showing the following two facts:\smallskip

\begin{enumerate}

\item Any point in $\partial X$ represented by a geodesic ray that lies inside some $K$--dis\-place\-ment set of $h$ is fixed by $h$.\smallskip

\item The centraliser $C_G(h)$ acts cocompactly on any $K$--displacement set of $h$.\smallskip

\end{enumerate}

\noindent
The `only if' direction is more involved. In this case we are given the information about the action on the boundary, and we need to extract the information about the action on the complex. For this we need $X$ to satisfy the following property. We say that $X$ is \emph{almost extendable} if there exists a constant $E \geqslant 0$ such that for every pair of vertices $x,y$ in $X$ there is a good geodesic ray issuing from $x$ and passing within distance $E$ from $y$. The following theorem is also of independent interest.

\begin{wtw}[Theorem~\ref{tw:almostext}]\label{tw:twbkasza}
Let $X$ be a noncompact systolic complex, on which a group $G$ acts geometrically. Then $X$ is almost extendable.
\end{wtw}

The proof of this theorem relies on the study of topology at infinity of systolic complexes. It is similar to the proof of an analogous theorem for $\mathrm{CAT}(0)$ spaces \cite{Ont}. The main difference is that our proof uses the notion of connectedness at infinity, whereas the one in \cite{Ont} uses cohomology with compact supports. The key fact is that in the setting above, the complex $X$ is not $1$--connected at infinity (see \cite{Osa}).\medskip

In the second part of the article we consider arbitrary hyperbolic isometries of $X$ (not necessarily the virtually central ones). One can still ask whether such an isometry $h$ has any fixed points in $\partial X$. In the setting of $\mathrm{CAT}(0)$ spaces, a hyperbolic isometry $h$ has an \emph{axis}, that is, an $h$--invariant geodesic line, and this axis determines two fixed points of $h$ in $\partial X$. 
In our situation $h$ also has a kind of axis (see \cite{E2}), but unfortunately this axis does not have to determine an $h$--invariant good geodesic. 
 In fact, an $h$--invariant good geodesic may not exist. However, we do prove that $h$ has a pair of fixed points in $\partial X$. \bigskip

\begin{wtw}[Proposition~\ref{prop:fixedpointsconvergence}]\label{tw:twc}
Let $G$ be a group acting geometrically on a systolic complex $X$ and let $h \in G$ be a hyperbolic isometry. Then:

\begin{enumerate}

		\item there exist points $h^{-\infty}$ and $ h^{+\infty}$ in the boundary $\partial X$ which are fixed by $h$,

			\item for any vertex $x \in X$ we have $(h^n\cdot x)_n \to h^{+\infty}$ and $(h^{-n} \cdot x)_n \to h^{-\infty}$ as $n \to \infty$ in the compactification $\overline{X} = X \cup \partial X$.
\end{enumerate}
\end{wtw}
The second statement shows that $h^{+\infty}$ and $h^{-\infty}$ are in a certain sense the canonical fixed points of $h$. To find $h^{+\infty}$ and $h^{-\infty}$ we construct an `almost $h$--invariant' good geodesic in $X$, by which we mean a geodesic that is contained in some $K$--displacement set of $h$. This requires analysing the construction of good geodesics.  We go through the steps of the construction and show that given any two vertices $x$ and $y$ in the minimal displacement set of $h$, a good geodesic between $x$ and $y$ is contained in a $K$--displacement set where $K$ is independent of distance between $x$ and $y$. Then we construct a bi-infinite good geodesic as a limit of finite good geodesics contained in the $K$--displacement set.

In order to prove the second part of Theorem~\ref{tw:twc} we also study good geodesics contained in the flats of $X$. In particular, we give a simple criterion for a geodesic contained in a flat to be a good geodesic and we show that any geodesic that is good in the flat is also good in the complex $X$.\smallskip

We believe that the results presented in this article may be used in the further study of systolic groups via their boundaries. Theorems~\ref{tw:twa} and \ref{tw:twc} are the first steps in analysing the dynamics of the action of $h$ on $\partial X$, which in the $\mathrm{CAT}(0)$ setting plays the key role in e.g., \cite{PaSw}, where the topology of $\partial X$ is related to splittings of $G$ over $2$--ended subgroups. Theorem~\ref{tw:twbkasza} seems to be of a more general nature; its $\mathrm{CAT}(0)$ counterpart has been widely used in the study of $\mathrm{CAT}(0)$ groups and boundaries.

\subsubsection*{Organisation}
The article consists of an introductory Section~\ref{sec:preliminaries}, where we give background on systolic complexes and boundaries, and of the two main parts. In the first part, which occupies Sections \ref{sec:almostext} and \ref{sec:trivialact}, we prove Theorem~\ref{tw:twbkasza} and after establishing basic facts about $K$--displacement sets we give a proof of Theorem~\ref{tw:twa}. In the second part (Sections~\ref{sec:constructgood} and \ref{sec:fixedpoints}) we first sketch the construction of good geodesics, and then we prove Theorem~\ref{tw:twc}.

\subsubsection*{Acknowledgements}I would like to thank Piotr Przytycki for suggesting the problem and for many helpful discussions. I would also like to thank Damian Osajda for helpful discussions and comments, in particular leading to the proof of Theorem~\ref{tw:twbkasza}. Finally, I thank Dieter Degrijse and Damian Osajda for a careful proofreading of the manuscript and help in improving the exposition.

The author was supported by the Danish National Research Foundation through the Centre for Symmetry and Deformation (DNRF92) and by the EPSRC First Grant EP/N033787/1.

\section{Systolic complexes and their boundaries}\label{sec:preliminaries}

In this section we give some background on systolic complexes and their boundaries. We also fix the terminology and notation that is used throughout the article.
\subsection{Systolic simplicial complexes.}

Let $X$ be a simplicial complex. We assume that $X$ is finite dimensional and uniformly locally finite, i.e., there is a uniform bound on the degree of vertices in $X$. We equip $X$ with the CW--topology, and always treat it as a topological space (we do not make a distinction between an abstract simplicial complex and its geometric realisation). Let $X^{(k)}$ denote the $k$--skeleton of $X$. In particular $X^{(0)}$ is the vertex set of $X$. For any subset $A \subset X^{(0)}$, a subcomplex \emph{spanned} by $A$ is the largest subcomplex of $X$ that has $A$ as its vertex set. We denote this subcomplex by $\mathrm{span}(A)$. A map $f \colon X \to Y$ of simplicial complexes is \emph{simplicial} if $f(X^{(0)}) \subset Y^{(0)}$ and whenever vertices $x_0, x_1, \ldots, x_n$ span a simplex of $X$ then their images $f(x_0), f(x_1), \ldots,  f(x_n)$ span a simplex of $Y$. Note that a simplicial map is continuous, and in particular a simplicial automorphism is a homeomorphism of $X$.\medskip

In this article we will be particularly interested in metric aspects of simplicial complexes.

\begin{de}
We endow the vertex set $X^{(0)}$ with a metric, where the distance $d(x,y)$ between vertices $x$ and $y$ is defined to be the combinatorial distance in the $1$--skeleton, i.e., the minimal number of edges of an edge--path joining $x$ and $y$. 

For two subcomplexes $A,B \subset X$, we define the distance $d(A,B)$ to be the minimal distance between vertices $a \in A$ and $b \in B$.\end{de} 

Whenever we refer to the \emph{metric} on $X$ we mean the metric on $X^{(0)}$ defined above. Consequently, a \emph{geodesic} in $X$ is a sequence of vertices $(v_0, v_1, \ldots,  v_n)$ such that for any $0 \leqslant i, j \leqslant n$ we have $d(v_i, v_j) =|j-i|$. Analogously we define a geodesic that is infinite in one or both ends. In the first case we call it a \emph{geodesic ray}, in the second case: a \emph{geodesic line} or a \emph{bi-infinite geodesic}. Observe that a simplicial map is $1$--Lipschitz and any simplicial automorphism is an isometry of $X$.\medskip

We now briefly recall the notions needed to define systolic complexes. We say that $X$ is \emph{flag} if every set of vertices of $X$ pairwise connected by edges spans a simplex of $X$. A flag simplicial complex is completely determined by its $1$--skeleton $X^{(1)}$ or, equivalently, by its vertex set $X^{(0)}$ and the metric $d$ defined above. For a vertex $v \in X$ the \emph{link} of $v$ is a subcomplex $\link{v}{X}$ of $X$, that consists of all the simplices of $X$ that do not contain $v$, but together with $v$ span a simplex of $X$. A \emph{cycle} in $X$ is the image of a simplicial map $ f \colon S^1 \to X$ from the triangulation of the $1$--sphere to $X$. A cycle is \emph{embedded} if $f$ is injective. Let $\gamma$ be an embedded cycle. The \emph{length} of $\gamma$, denoted by $|\gamma|$ is the number of edges of $\gamma$. A \emph{diagonal} of $\gamma$ is an edge in $X$ that connects two nonconsecutive vertices of $\gamma$.\smallskip

We are ready to define systolic complexes. Our main reference for the theory of systolic complexes is \cite{JS2}.

\begin{de}Given a natural number $k \geqslant 6$, a simplicial complex $X$ is $k$--\emph{large} if every embedded cycle $\gamma $ in $ X$ with  $4 \leqslant |\gamma| < k$ has a diagonal. We say that $X$ is $\infty$--\emph{large} if it is $k$--large for every $k \geqslant 6$. \end{de}

\begin{de}A simplicial complex $X$ is $k$--\emph{systolic} if it is simply connected and if for every vertex $v \in X$ the link $\link{v}{X}$ is flag and $k$--large. If $k=6$ then we abbreviate $6$--\emph{systolic} to \emph{systolic}.
\end{de}

A $k$--systolic complex is flag and $k$--large \cite[Proposition 1.4 and Fact 1.2(4)]{JS2}. Note that if $k \leqslant m$ then `$m$--systolic' implies `$k$--systolic'. In this article we will be interested in the (most general) case of $k=6$. This case is of particular importance in the theory, as for $k \geqslant 7$ one shows that $k$--systolic complexes are $\delta$--hyperbolic \cite[Theorem 2.1]{JS2}.\medskip

The condition of $k$--largeness, when applied to the link of a vertex $ v \in X$, serves as a~kind of upper bound for the curvature around $v$. In particular complexes with $6$--large links are called complexes of \emph{simplicial nonpositive curvature} (SNPC). Consequently, systolic complexes can be thought of as simplicial analogues of $\mathrm{CAT}(0)$ metric spaces.

\begin{de}Let $v \in X$ be a vertex and let $n$ be a positive integer. Define the \emph{ball of radius $n$ centred at} $v$ by $ B_n(v,X) = \mathrm{span}\{x \in X^{(0)} \mid d(x,v) \leqslant n\} \subset X.$ Define the \emph{sphere of radius $n$ centred at} $v$ by $ S_n(v,X) = \mathrm{span}\{x \in X^{(0)} \mid d(x,v) = n\}.$ For a subcomplex $A \subset X$ define the \emph{ball of radius $n$ around} $A$ by \[B_n(A,X) = \bigcup_{v \in A^{(0)}} B_n(v,X).\] We also refer to $B_n(A,X)$ as an $n$--\emph{neighbourhood of} $A$ in $X$.
\end{de}

A subcomplex $A \subset X$ is \emph{convex} if for every two vertices $x,y \in A$ any geodesic between $x$ and $y$ in $X$ is contained in $A$. Note that since geodesics in $X$ are not necessarily unique, a subcomplex $A \subset X$ can be isometrically embedded and not convex.

\begin{prop}\label{prop:ballcontractible}Let $X$ be a systolic complex. Then the following hold:
\begin{enumerate}
\item \label{itemlabel:ballscontract} For any convex subcomplex $A \subset X$ the ball $B_n(A,X)$ is convex and contractible \cite[Corollary 7.5]{JS2}. In particular for any vertex $v \in X$ the ball $B_n(v,X)$ is convex and contractible.

	\item The complex $X$ is contractible \cite[Theorem 4.1(1)]{JS2}.
\end{enumerate}
\end{prop}

We finish this section with some terminology regarding group actions on simplicial complexes. Let $G$ be a (discrete) group acting on a simplicial complex $X$. We assume that $G$ acts via simplicial automorphisms. We say that the action is:\begin{itemize}
\item \emph{proper} if for every vertex $v \in X$ the stabiliser $G_v$ is finite,
\item \emph{cocompact} if there exists a compact subset $K \subset X$ such that $G \cdot K =X$,
\item \emph{geometric} if it is proper and cocompact.
\end{itemize}
A group is called \emph{systolic} if it acts geometrically on a systolic complex.

\subsection{Boundaries of systolic complexes.}\label{subsec:boundaries}
Given a (noncompact) systolic complex $X$ one can define the \emph{boundary at infinity} (or shortly the \emph{boundary}) $\partial X$ of $X$. Analogously to the cases of $\delta$--hyperbolic and $\mathrm{CAT}(0)$ spaces, the boundary for systolic complexes is given by a set of equivalence classes of geodesic rays, such that asymptotic rays are equivalent. 
In this section we give the definition of the boundary and briefly discuss its key features that are needed in this article. For more details we refer the reader to \cite{OP}.

The main difference from $\delta$--hyperbolic and $\mathrm{CAT}(0)$ cases is that, instead of arbitrary geodesic rays, to define the boundary one uses a canonically defined subcollection of the so-called \emph{good geodesic rays}. To define good geodesic rays one first defines \emph{good geodesics}. The actual definition of good geodesics (which is quite involved) is needed only in Section~\ref{sec:fixedpoints}, and therefore we give this definition in Section~\ref{sec:constructgood}. 

In order to follow the arguments in Sections~\ref{sec:almostext} and \ref{sec:trivialact} it is enough to know that a good geodesic is a certain geodesic in $X$, and that the subclass of good geodesics has the following properties: 

\begin{enumerate} 
\item for any two vertices there exists a (not necessarily unique) good geodesic joining these vertices,
\item any subgeodesic of a good geodesic is a good geodesic,
\item \label{item:ggeoaregeometric} any simplicial automorphism of $X$ maps good geodesics to good geodesics.
\end{enumerate}

A \emph{good geodesic ray} is a geodesic ray, such that any of its finite subgeodesics is a good geodesic. By~(\ref{item:ggeoaregeometric}) any simplicial automorphism of $X$ maps good geodesic rays to good geodesic rays.

 Let $\mathcal{R}$ denote the set of all good geodesic rays in $X$. For a vertex $O \in X$ let $\mathcal{R}_{O}$ denote the set of all good geodesic rays starting at $O$.

\begin{de}\label{def:boundary}Let $X$ be a systolic complex. Define the \emph{boundary of} $X$ to be the set $\partial X = \mathcal{R} / \sim$ where for rays $\eta=(v_0, v_1, \ldots)$ and $\xi=(w_0, w_1, \ldots)$ we have $\eta \sim \xi$ if and only if there exists $K \geqslant 0$, such that for every $i \geqslant 0$ we have $d(v_i, w_i) \leqslant K$.

Define the \emph{boundary of $X$ with respect to the basepoint} $O$ to be the set $\partial_OX	=\mathcal{R}_O / \sim $, where $\sim$ is the same equivalence relation as above. In both cases let $[\eta]$ denote the equivalence class of $\eta$.
\end{de}
  
For any vertex $O \in X$ there is a bijection $\partial X \to \partial_OX$ \cite[Corollary 3.10]{OP}. In particular this means that for every geodesic ray $\eta \subset X$ and for every vertex $O \in X$ there is a ray $\xi \subset X$ starting at $O$ such that $[\eta]= [\xi]$ in $\partial X$. This fact will be used many times in this article. The set $\overline{X}= X \cup \partial_OX$ can be equipped with a topology that extends the standard topology on $X$, and turns $\overline{X}$ into a compact topological space \cite[Propositions 4.4 and 5.3]{OP}. For any two vertices $O, O' \in X$ there is a homeomorphism between $X \cup \partial_OX$ and $X \cup \partial_{O'}X$ \cite[Lemma 5.5]{OP}. Any simplicial action of a group on $X$ extends to an action by homeomorphisms on $\overline{X}$ \cite[Theorem A(4)]{OP}. 

In this article we will mostly be concerned with the induced action on the boundary, not on the entire $\overline{X}$. Moreover, we will be interested in the action on the boundary seen as a set, not as a topological space. For this, we can use a slightly simpler definition. 

\begin{de}\label{def:actionboundary}
Suppose that a group $G$ acts simplicially on $X$. We define an action of $G$ on the set $\partial X$ as follows. Let $[\eta] \in \partial X$ where $ \eta=(v_0, v_1, \ldots )$. Then define $g \cdot [\eta] =[g\cdot \eta] $ where $ g \cdot \eta = (g \cdot v_0, g \cdot v_1, \ldots)$. It is straightforward to check that this is well defined and it defines an action of $G$ on $\partial X$.
\end{de}

One can also verify, that via the bijection $\partial X \to \partial_OX$ the action described above agrees with the action on $\partial_OX$ defined in \cite{OP}.\medskip

We conclude this section with certain metric properties of good geodesics. The following is a crucial property, which can be seen as a coarse version of $\mathrm{CAT}(0)$ inequality for good geodesics.

\begin{tw}\cite[Corollary 3.4]{OP}
\label{tw:contracting}
Let $(v_0, v_1, v_2, \ldots, v_n)$ and $(w_0, w_1, w_2,\allowbreak \ldots ,\allowbreak w_m)$ be good geodesics in a systolic complex $X$ such that $v_0=w_0$. Then for any $0 \leqslant c \leqslant 1$ we have
\begin{equation*} d(v_{\ent{cn}},w_{\ent{cm}} ) \leqslant c \cdot d(v_n, w_m) +D,
\end{equation*}
where $D$ is a universal constant. 
\end{tw}

This leads to the following corollary, which will be very useful to us.

\begin{cor}
\label{coro:doublecontracting}
Let $ \eta = (v_0, v_1, v_2, \ldots)$ and $\xi =(w_0, w_1, w_2, \ldots)$ be good geodesic rays in a systolic complex $X$. If $[\eta]=[\xi]$ in $\partial X$ then for every $i \geqslant 0$ we have
\begin{equation*}
d(v_i, w_i) \leqslant d(v_0, w_0) +2D+1,
\end{equation*}
where $D$ is the constant appearing in Theorem~\ref{tw:contracting}.
\end{cor}

\begin{proof}Since $[\eta]=[\xi]$, there is a constant $K \geqslant 0$ such that for all $i$ we have $d(v_i, w_i) \leqslant K$. Fix $i\geqslant 0$, pick $n> K$ and let $ z=(z_0=v_0, z_1, z_2, \ldots , w_{ni})$ be a good geodesic joining $v_0$ and $w_{ni}$. By Theorem~\ref{tw:contracting} applied to $(v_0, v_1, v_2, \ldots, v_{ni})$ and $z$ we have 
\begin{equation}\label{eq:doublecontract} d(v_i, z_i) \leqslant \frac{1}{n}d(v_{ni}, w_{ni}) +D \leqslant \frac{K}{n} +D \leqslant 1+ D.
\end{equation}

Applying Theorem~\ref{tw:contracting} to $z$ and $(w_0, w_1, w_2, \ldots, w_{ni})$ (with the direction reversed) we obtain that $d(z_i, w_i) \leqslant d(v_0, w_0) +D$. This, together with \eqref{eq:doublecontract} and the triangle inequality gives the claim. \end{proof}

\section{Almost extendability of systolic complexes}\label{sec:almostext}

In this section we study a property of  metric spaces called the \emph{almost extendability}. This property can be defined for arbitrary geodesic metric spaces. The definition we present is adjusted to the setting of systolic complexes.

\begin{de}\label{def:almostext}
A systolic complex $X$ is \emph{almost extendable}, if there exists a constant $E \geqslant 0$ such that for any two vertices $x$ and $y$ of $X$, there is a good geodesic ray starting at $y$ and passing within distance $E$ from $x$.
\end{de}

 It is easy to construct systolic complexes (in fact, trees) that are not almost extendable. For example, let $T$ denote the half-line $\mathbb{R}_+$ with the interval of length $n$ attached to every integer $n \in \mathbb{R}_+$. The standard triangulation turns $T$ into a systolic complex in which every combinatorial geodesic is a good geodesic. One can easily see that $T$ is not almost extendable. When we equip a systolic complex with a geometric action of a group then the situation changes.

\begin{tw}
\label{tw:almostext}Let $X$ be a noncompact systolic complex, on which a group $G$ acts geometrically. Then $X$ is almost extendable.
\end{tw}

The analogous theorem is true in the $\mathrm{CAT}(0)$ setting \cite[Theorem B]{Ont}, and  it is an exercise in the setting of $\delta$--hyperbolic groups (see \cite{Ont}). Our proof is similar to the one for $\mathrm{CAT}(0)$ spaces, however, it can be seen as more direct. The main difference is that instead of cohomology with compact supports, our proof uses the notion of connectedness at infinity. We begin by recalling this notion.

\begin{de}Let $Y$ be a topological space and let $n \geqslant -1$ be an integer. We say that $Y$ is \emph{$n$--connected at infinity} if for every $ -1 \leqslant k \leqslant n$ the following condition holds: for every compact set $K \subset Y$ there exists a compact set $L \subset Y$ such that $K \subset L$ and every map $S^k = \partial B^{k+1} \to Y \setminus L$ extends to a map $B^{k+1} \to Y \setminus K$. 

For $k=-1$ we define $S^{-1}= \emptyset$ and $B^0=\{\ast\}$. In particular $Y$ is $(-1)$--connected at infinity if and only if it is not compact.
\end{de}

Note that if $Y$ is a simplicial complex then (in view of the Simplicial Approximation Theorem) in the above definition it is enough to consider only simplicial maps.\medskip

The following theorem of D. Osajda is the crucial ingredient in the proof of Theorem~\ref{tw:almostext}.

\begin{tw}\cite[Theorem 3.2]{Osa}
\label{tw:sysnot1conn} Let $X$ be a noncompact systolic complex, on which a group $G$ acts geometrically. Then $X$ is not $1$--connected at infinity. 
\end{tw}

\begin{proof}[Proof of Theorem~\ref{tw:almostext}.]
First we show that it is enough to prove the following claim.

\begin{claim1} Let $p$ be a fixed vertex. Then there exists a constant $E'$ such that, for any $g \in G$ there is a good geodesic ray starting at $p$ and passing within $E'$ from $g \cdot p$.
\end{claim1}

Indeed, let $x$ and $y$ be arbitrary vertices of $X$. By cocompactness there exists $R > 0$ and elements $g_1, g_2 \in G$ such that we have $d(g_1 \cdot p, x) \leqslant R$ and $d(g_2 \cdot p, y) \leqslant R$. By Claim 1 there exists a good geodesic ray $\eta$ starting at $p$ and passing within $E'$ from $g_1^{-1}g_2 \cdot p$. Then the ray $g_1 \cdot \eta$ starts at $g_1 \cdot p$ and passes within $E'$ from $g_2 \cdot p$, and hence it passes within $E'+R$ from $y$. Now let $\xi$ be a good geodesic ray starting at $x$ and such that $[\xi]=[g_1 \cdot \eta]$.  Write $g_1 \cdot \eta = (g_1 \cdot p=v_0, v_1, \ldots)$ and $\xi=( x=w_0, w_1, \ldots)$. Then
by Corollary~\ref{coro:doublecontracting} for every $i \geqslant 0$ we have \[d(v_i, w_i) \leqslant d(g_1 \cdot p, x) + 2D +1 \leqslant R+2D+1.\] Since $g_1 \cdot \eta$ passes within $E' +R$ from $y$, we have that $\xi$ passes within $E'+R+R+2D+1$ from $y$. Therefore Claim 1 implies the theorem (with constant $E = E'+R+R+2D+1$).\medskip

The rest of the proof is devoted to proving Claim 1. We need a little preparation. In what follows, for a good geodesic or a good geodesic ray $\eta$ we will denote its vertices by $\eta(i)$, for $i \geqslant 0$, i.e.,\ $\eta=(\eta(0), \eta(1), \eta(2), \ldots)$. In other words, the geodesic $\eta$ may be seen as a map $\mathbb{N} \to X$. We still treat $\eta$ as a subset of $X$; the above notation is introduced only to simplify the exposition. 

For a good geodesic $\eta$ let $l_{\eta}$ denote the supremum of natural numbers $l$, such that $\eta$ can be extended to a good geodesic on the interval $[0, l] = \{0,1, \ldots, l\} \subset \mathbb{N}$. Note that $l_{\eta}$ does not have to be attained, in that case we write $l_{\eta} = \infty$. Observe that if $l_{\eta} < \infty$ then there is an extension of $\eta$ to the interval $[0, l_{\eta}]$. If $l_{\eta}= \infty$ then by the fact that $X \cup \partial X$ is compact, there is an extension of $\eta$ to the interval $[0, \infty)$, i.e.,\ the geodesic $\eta$ can be extended to a good geodesic ray. For vertices $x, y \in X$ let $[\![ x,y ]\!]$ denote a good geodesic between these two vertices. Note that such a geodesic in general is not unique.

Now we begin the proof of Claim 1. We proceed by contradiction. Assume that Claim 1 does not hold, then we have the following:

($\ast$) For every $r >0$ there exists $g_r \in G$ such that for every vertex $x \in B_r(g_r \cdot p, X)$ we have $l_{[\![p,x]\!]} < \infty$ for every good geodesic $[\![p,x]\!]$.

\begin{claim2} For every $r>0$ we have \[\mathrm{sup} \{   l_{[\![p,x]\!]} \mid   [\![p,x]\!] \text{ where } x \in B_r(g_r \cdot p, X)\} < \infty.\]

(The supremum is taken over all possible good geodesics that start at $p$ and end at a vertex of $B_r(g_r \cdot p, X)$.)
\end{claim2}

To prove Claim 2, assume conversely that there exists a sequence of good geodesics $([\![p, x_i]\!])_i$ with $x_i \in B_r(g_r \cdot p, X)$, such that $l_{[\![p, x_i]\!]} \rightarrow \infty$ as $n \rightarrow \infty$.
Let $\eta_i$ denote a good geodesic extending $[\![p, x_i]\!]$ to the interval $[0, l_{[\![p, x_i]\!]}]$ (we choose one for each $i$). 
Using a diagonal argument, out of the sequence $(\eta_i)_i$ one constructs an infinite geodesic ray $\xi$ that issues from $p$, and such that for any interval $[0,l]$ we have $\xi\big|_{[0,l]} = \eta_i$ for some $i=i_l$ (cf.\ \cite[Proposition 5.3]{OP}). In particular, the ray $\xi$ intersects the ball $B_r(g_r \cdot p, X)$, which contradicts ($\ast$).

\begin{claim3} For every $r>0$ there exists $r'>r$ such that for every vertex $y \in X \setminus B_{r'}(p, X)$, every good geodesic $[\![p, y]\!]$ misses the ball $B_r(g_r \cdot p, X)$, i.e.,\ we have \[ [\![p, y]\!] \cap B_r(g_r \cdot p, X)= \emptyset. \]
\end{claim3}

Note that we have $p \notin B_{r}(g_r \cdot p, X)$, for otherwise we would get a contradiction with ($\ast$) as there always is a geodesic ray issuing from $p$ (since $X$ is noncompact). Let $ r'=\mathrm{sup} \{   l_{[\![p,x]\!]} \mid   [\![p,x]\!] \text{ where } x \in B_r(g_r \cdot p, X)\}.$ Then the claim follows from the definition of $r'$.

\begin{claim4} The complex $X$ is $1$--connected at infinity.
\end{claim4}
First observe that since $X$ is noncompact, it is $(-1)$--connected at infinity. Let $ K \subset X$ be a compact subset. Take $M>0$ such that $K \subset B_M(p, X)$ and consider the ball $B_{M+D+2}(p, X)$, where $D$ is the constant appearing in Theorem~\ref{tw:contracting} (the reason why we need to pass to the larger ball will become clear later on). 

Pick $r> M+D+2$. By Claim 3  (after `translating its statement by $g_r^{-1}$') there exists $r'>r$ such that every good geodesic joining a vertex $y \in X \setminus B_{r'}(g_r^{-1}\cdot p, X) $ with $g_r^{-1} \cdot p$, misses the ball $B_r(p,X)$. Set $L =B_{r'}(g_r^{-1}\cdot p, X)$. By construction we have $K \subset L$, and $g_r^{-1} \cdot p \in L \setminus K$. 

Suppose $f \colon S^0 \to X\setminus L$ is a simplicial map. Let $v_1$ and $v_2$ be the two vertices in the image of $f$. For $i \in \{0,1\}$ let $\eta_i$ be a good geodesic joining $g_r^{-1} \cdot p$ with $v_i$. Both $\eta_i$ miss the ball $B_r(p,X)$ (and hence they miss $K$) and therefore their union defines a map $F \colon B^1 \to X \setminus K$ that extends $f$. This shows that $X$ is $0$--connected at infinity.

Now let $f \colon S^1 \to X \setminus L$ be a simplicial map. Let $(v_0, v_1, \ldots, v_n, v_{n+1}= v_0)$ be the vertices of the image of $f$ appearing in this order, i.e.,\ for all $i$ vertices $v_i$ and $v_{i+1}$ are adjacent. For every $i \in \{ 0,\ldots, n \}$ let $\eta_i$ be a good geodesic joining $g_r^{-1} \cdot p$ and $v_i$. Observe that no $\eta_i$ intersects the ball $B_r(p,X)$. We will use $\eta_i$'s to construct the required extension of $f$ to the disk $B^2$. 

For any $i$ consider the cycle $\alpha_i \subset X$ which is the union \[\alpha_i= \eta_i \cup \eta_{i+1} \cup [v_i,v_{i+1}].\] We will show that $\alpha_i$ can be contracted to a point in its $(D+2)$--neighbourhood. First note that either $\eta_i$ and $\eta_{i+1}$ have the same length, or their lengths differ by $1$. 

In the first case put $k= d(g_r^{-1}\cdot p, v_i)= d(g_r^{-1}\cdot p, v_{i+1})$. Since $\eta_i$ and $\eta_{i+1}$ start at the same vertex and end at vertices that are adjacent, it follows from Theorem~\ref{tw:contracting} that for any $j \in \{0, 1, \ldots, k\}$ we have \[d(\eta_i(j), \eta_{i+1}(j)) \leqslant D+1.\] 
For any $j \in \{0, \ldots, k\}$ let $\beta^i_j$ be a geodesic between $\eta_i(j)$ and $\eta_{i+1}(j)$ (note that $\beta^i_0 $ is the vertex $ g_r^{-1} \cdot p$ and $\beta^i_k$ is the edge $[v_i, v_{i+1}]$). Now for every $j \in \{0, \ldots, k-1\}$ consider a cycle $\gamma^i_j$ defined as \[\gamma^i_j = \beta^i_j \cup [\eta_{i+1}(j), \eta_{i+1}(j+1)] \cup \beta^i_{j+1} \cup [\eta_{i}(j), \eta_{i}(j+1)].\] By construction $\gamma^i_j$ is contained in the ball $B_{D+2}(\eta_{i}(j), X)$ and therefore it can be contracted inside $B_{D+2}(\eta_{i}(j), X)$, as balls in $X$ are contractible (see Proposition~\ref{prop:ballcontractible}.(\ref{itemlabel:ballscontract})). These contractions of $\gamma^i_j$ for all $j \in \{0, \ldots, k\}$ form a contraction of $\alpha_i$ inside the ball $B_{D+2}(\eta_i,X)$ around the geodesic $\eta_i$. (Formally, by a \emph{contraction} we mean a simplicial map from a simplicial disk $f \colon B^2 \to B_{D+2}(\eta_i,X)$ such that $f$ maps the boundary $\partial B^2$ isomorphically onto $\alpha_i$.)

In the second case, assume that $\eta_{i+1}$ is longer than $\eta_i$, i.e., we have $d(g_r^{-1}\cdot p, v_i)=k$ and $d(g_r^{-1} \cdot p, v_{i+1})=k+1$. In this case the concatenation $\eta_{i} \cup [v_i,v_{i+1}]$ is a geodesic. Then it follows from \cite[Lemma 7.7]{JS2} that $\eta_{i+1}(k)$ and $v_{i}$ are adjacent, and therefore $\eta_{i+1}(k), v_i$ and $v_{i+1}$ span a $2$--simplex. Now $\eta_{i}$ and $\eta_{i+1}\big|_{[0,k]}$ are of the same length and their endpoints $v_i$ and $\eta_{i+1}(k)$ are adjacent. Proceeding as in the first case we obtain a contraction of the cycle \[\eta_i \cup \eta_{i+1}{\big|}_{[0,k]} \cup [v_i, \eta_{i+1}(k)]\] inside the ball $B_{D+2}(\eta_i,X)$. Adding the $2$--simplex $[v_i, \eta_{i+1}(k), v_{i+1}]$ we obtain the desired contraction of $\alpha_i= \eta_i \cup \eta_{i+1} \cup [v_i,v_{i+1}]$.\medskip

Finally, contractions of $\alpha_i$ for all $i\in \{0, \ldots, n\}$ form the contraction of $(v_0, \ldots v_n, v_0)$ that is performed in the $(D+2)$--neighbourhood of the union of all $\eta_i$'s. Since every $\eta_i$ misses the ball $B_r(p,X)$, the $(D+2)$--neighbourhood of $\eta_i$ misses the ball $B_M(p, X)$ as $r >M+D+2$, and hence it misses $K$ as $K \subset B_M(p,X)$. We conclude that the constructed contraction of $(v_0, \ldots, v_n, v_0)$ defines the extension of $f$ that misses $K$. This finishes the proof of Claim 4. \medskip

This gives a contradiction with Theorem~\ref{tw:sysnot1conn} and hence proves Claim 1.
\end{proof}

\section{Isometries acting trivially on the boundary}\label{sec:trivialact}


In this section, given a group $G$ acting geometrically on a systolic complex $X$, we investigate which elements of $G$ act trivially on the boundary $\partial X$. Before proving the main theorem which characterises such elements in terms of their centralisers in $G$, we introduce the terminology and briefly discuss the tools needed in the proof.

\subsection{Hyperbolic isometries and their $K$--displacement sets.}
\label{subsec:minsets}

Let $h$ be an isometry (i.e.,\ a simplicial automorphism) of a systolic complex $X$. We say that $h$ is \emph{hyperbolic} if it does not fix any simplex of $X$. If $h$ is hyperbolic, then any of its powers is hyperbolic as well (\cite{E2}). To such $h$ one associates the \emph{displacement function} $d_h\colon X^{(0)} \to \mathbb{N}$ defined as $d_h(x)= d(x,h\cdot x)$. The minimum of $d_h$ (which is always attained) is called the \emph{translation length} of $h$ and is denoted by $\trol{h}$.

\begin{de}
\label{def:minset}Let $h$ be a hyperbolic isometry of a systolic complex $X$. The \emph{minimal displacement set} $\minset{h}$ is the subcomplex of $X$ spanned by all the vertices of $X$ which are moved by $h$ the minimal distance, i.e.:
\[\minset{h} = \mathrm{span}\{x \in X^{(0)} \mid d(x, h \cdot x) = \trol{h}\}.\]
More generally, for a natural number $K \geqslant \trol{h}$ define the \emph{$K$--displacement set} as
\[\disp{K}{h} = \mathrm{span}\{x \in X^{(0)} \mid d(x, h \cdot x) \leqslant K\}.\]
Clearly we have $\disp{K}{h} \subset \disp{K'}{h}$ for $K \leqslant K'$ and $\disp{\trol{h}}{h}=\minset{h}$. 
\end{de}
Let us mention that $\minset{h}$ is a systolic complex on its own, and its inclusion into $X$ is an isometric embedding (\cite{E2}). We do not know whether the same is true for $\disp{K}{h}$ for $K > \trol{h}$. In this article we are interested only in the coarse-geometric behaviour of $\disp{K}{h}$.

Observe that if $x \in X$ is a vertex such that $d(x, \disp{K}{h}) \leqslant C$ for some $C \geqslant 0 $, then by the triangle inequality we have $d(x, h \cdot x) \leqslant{K+2C}$. This means that $B_C(\disp{K}{h},X) \subseteq \disp{K+2C}{h}.$ In the presence of a geometric action of a group, the (partial) converse also holds.

\begin{lem}
\label{lem:disp} Let $G$ be a group acting geometrically on a systolic complex $X$ and suppose that $h \in G$ is a hyperbolic isometry. Pick $K\geqslant \trol{h}$. Then there exists $C >0$ such that for all $K' \leqslant K$ we have $\disp{K}{h} \subset B_C(\disp{K'}{h}, X)$. 
\end{lem}

The lemma is an easy consequence of the following theorem.

\begin{tw}
\label{tw:centraliser}
Let $G$ act geometrically on a systolic complex $X$, and let $h \in G$ be a hyperbolic isometry. Then for any $K \geqslant \trol{h}$ the centraliser $C_G(h) \subset G$ acts geometrically on the subcomplex $\disp{K}{h} \subset X$.
\end{tw}

\begin{proof}The proof is a verbatim translation of K. Ruane's proof of a similar result for $\mathrm{CAT}(0)$ spaces \cite[Theorem 3.2]{Kru}. The original proof treats only the case where $ \disp{K}{h} = \minset{h}$, however it is straightforward to check that it carries through for any $\disp{K}{h}$. We include the proof for the sake of completeness.

First we check that $C_G(h)$ leaves $\disp{K}{h}$ invariant. Let $x \in \disp{K}{h}$ be a vertex and take $g\in C_G(h)$. Then we have \[d(g\cdot x, hg \cdot x)=d(g\cdot x, gh \cdot x)=d(x, h \cdot x) \leqslant K\] and thus $g \cdot x \in \disp{K}{h}$.
Observe that the action of $C_G(h)$ on $\disp{K}{h}$ is proper, since the action of $G$ on $X$ is proper. We only need to check cocompactness. We proceed by contradiction. Assume that there is no compact subset of $\disp{K}{h}$ whose $C_G(h)$--translates cover $\disp{K}{h}$, and pick a vertex $x_0 \in \disp{K}{h}$. Then there exists a sequence of vertices $(x_n)_{n=1}^{\infty} $ of $ \disp{K}{h}$ such that $d(C_G(h) \cdot x_0, x_n) \rightarrow \infty$ as $n \rightarrow \infty$. Let $D \subset X$ be a compact set containing $x_0$ such that $G \cdot D = X$, and let $(g_n)_{n=1}^{\infty}$ be a sequence of elements of $G$ 
such that $g_n \cdot x_n \in D$. We can assume (by passing to a subsequence if necessary) that $g_n \neq g_m$ for $n \neq m$. Indeed, we have \[d(x_0, g_{n}^{-1} \cdot x_0) \geqslant d(x_0, x_n) -\mathrm{diam}D \rightarrow \infty \text{ as } n \rightarrow \infty.\]

Now consider the family of elements $\{g_nhg_n^{-1}\}_{n\geqslant 1}$ of $G$. We claim that the displacement functions $d_{g_nhg_n^{-1}}$ are uniformly bounded on $D$. Let $y \in D$ be a vertex. We have
\begin{equation*}
\begin{aligned}
d_{g_nhg_n^{-1}}(y) = {} & d(y, g_nhg_n^{-1} \cdot y) \\
					\leqslant {} & d(y, g_n \cdot x_n) + d(g_n \cdot x_n, g_nhg_n^{-1}(g_n \cdot x_n))\\
					 & + d(g_nhg_n^{-1}(g_n \cdot x_n), g_nhg_n^{-1} \cdot y)\\
					 = {} & d(y, g_n \cdot x_n)+ d(g_n \cdot x_n, g_nh \cdot x_n)+ d(g_nh \cdot x_n, g_nhg_n^{-1} \cdot y)\\
		  \leqslant {} & \mathrm{diam}(D) + K +\mathrm{diam}(D).
\end{aligned}
\end{equation*}

Since $G$ acts properly, it must be $g_nhg_n^{-1} = g_mhg_m^{-1}$ for $n \neq m$ (after passing to a subsequence). Therefore for all $n \neq m$ we have that $g_m^{-1}g_n \in C_G(h)$. Now for any $n \neq 1$ we get 
\begin{equation*}d(x_n, g_n^{-1}g_1 \cdot x_0) \leqslant d(x_n, g_n^{-1}g_1 \cdot x_1) + d(g_n^{-1}g_1 \cdot x_1 ,g_n^{-1}g_1 \cdot x_0) \leqslant \mathrm{diam}(D) + d(x_0, x_1). 
\end{equation*}

This gives a contradiction since $g_n^{-1}g_1 \in C_G(h)$, and by the choice of $x_n$ we have that $d(x_n, C_G(h) \cdot x_0) \rightarrow \infty$ as $n \rightarrow \infty$.\end{proof}

We are ready now to prove Lemma~\ref{lem:disp}.

\begin{proof}[Proof of Lemma~\ref{lem:disp}]
Pick any $K' \leqslant K$ and let $x_0 \in \disp{K'}{h}$ be a vertex. By Theorem~\ref{tw:centraliser} the centraliser $C_G(h)$ acts cocompactly on $\disp{K}{h}$. Hence there exists $R>0$ such that $\disp{K}{h} \subset C_G(h) \cdot B_{R}(x_0, X)$. Since $C_G(h) \cdot x_0 \subset \disp{K'}{h}$, taking $R$ as $C$ proves the lemma.
\end{proof}

\begin{rem}We believe that in Lemma~\ref{lem:disp} one can obtain a concrete distance estimate, i.e.,\ in the formula $\disp{K'}{h} \subset B_C(\disp{K}{h}, X)$ one can express $C$ as an explicit function of $K$ and $K'$. However, for our purposes, the existence of any constant $C$ is sufficient.
\end{rem}

\subsection{Trivial action on the boundary}\label{subsec:trivialact}
In this section we characterise hyperbolic isometries that act trivially on the boundary as being virtually central. More precisely, we show the following.

\begin{tw}\label{tw:virtuallycentral} Let $G$ be a group acting geometrically on a systolic complex $X$, and let $h \in G$ be a hyperbolic isometry. Then $h$ acts trivially on the boundary $\partial X$ if and only if the centraliser $C_G(h)$ has finite index in $G$.
\end{tw}

The theorem is a systolic analogue of a theorem of K. Ruane for $\mathrm{CAT}(0)$ spaces \cite[Theorem 3.4]{Kru}. In a certain way, our situation is more restrictive. Namely, by \cite[Corollary 5.8]{OsaPry} the centraliser $C_G(h)$ is commensurable with the product $F_n \times \mathbb{Z}$, where $F_n$ is the free group on $n$ generators for some $n \geqslant 0$. It follows that either of the assertions of Theorem~\ref{tw:virtuallycentral} holds true if and only if the group $G$ itself is commensurable with $F_n \times \mathbb{Z}$.

\begin{proof}[Proof of Theorem~\ref{tw:virtuallycentral}.]
\textbf{``if'' direction.} By Theorem~\ref{tw:centraliser} the centraliser $C_G(h)$ acts cocompactly on the minimal displacement set ${\minset{h}\!\subset X}$. Since the index $[G \colon C_G(h)]$ is finite, it follows that the action of $C_G(h)$ on $X$ is cocompact as well. Therefore there exists a constant $K\geqslant 0$ such that for any vertex $x \in X$, there is a vertex $y \in \minset{h}$ with $d(x,y) \leqslant K$. 
Hence, by the triangle inequality, for any $x \in X$  we have  $d(x, h \cdot x) \leqslant \trol{h} + 2K$. 

Now let $ \eta = (v_0, v_1, v_2, \ldots)$ be a good geodesic ray in $X$. For any $i \geqslant 0$ we have $d(v_i, h \cdot v_i) \leqslant\trol{h}+2K$ and hence $[\eta]= [h\cdot \eta]$ in $\partial X$.\medskip

\textbf{``only if'' direction.} Choose a vertex $y \in X$ and let $x \in X$ be an arbitrary vertex. By Theorem~\ref{tw:almostext} there exists a good geodesic ray $\eta=(v_0, v_1, v_2, \ldots)$ such that $v_0=y$ and for some $i \geqslant 0$ we have $d(v_i, x)\leqslant E$, where $E$ is a constant independent of $x$ and $y$. 

The isometry $h$ acts trivially on the boundary, so we have $[\eta]= [h \cdot \eta]$. Applying Corollary~\ref{coro:doublecontracting} we obtain \[d(v_i, h \cdot v_i) \leqslant d(v_0, h \cdot v_0) +2D +1,\] where $D$ is the constant appearing in Theorem~\ref{tw:contracting}.
In other words, we have $v_i \in \disp{K}{h}$ where $K= d(v_0, h \cdot v_0) +2D +1$. Since $d(v_i,x) \leqslant E$, the triangle inequality implies that $x \in \disp{K+2E}{h}$ (see the discussion after Definition~\ref{def:minset}).  Because $x$ was arbitrary, we have $X = \disp{K+2E}{h}$. By Theorem~\ref{tw:centraliser} the centraliser $C_G(h)$ acts cocompactly on $X$ and so it has finite index in $G$.
\end{proof}

We obtain the following corollary.

\begin{cor}\label{coro:entgroupactstriv} Let $G$ be a torsion-free group, acting geometrically on a systolic complex $X$. Then $G$ acts trivially on $\partial X$ if and only if $G \cong \mathbb{Z}$ or $G \cong \mathbb{Z}^2$.
\end{cor}
\begin{proof} If $G$ is isomorphic to either $\mathbb{Z}$ or $\mathbb{Z}^2$ then every element of $G$ is central, and by Theorem~\ref{tw:virtuallycentral} it acts trivially on $\partial X$.

Now assume that $G$ acts trivially on $\partial X$. Since $G$ is torsion-free, all of its elements are hyperbolic, and therefore by Theorem~\ref{tw:virtuallycentral} every element is virtually central.
Pick $h \in G$. By \cite[Corollary 5.8]{OsaPry} the centraliser $C_G(h)$ contains a finite-index subgroup $ H \cong F_n \times \mathbb{Z}$, where $F_n$ is the free group on $n$ generators for some $n \geqslant 0$. We must have $n \leqslant 1$ for otherwise no non-trivial element of $F_n$ would be virtually central in $G$. This means that $H$ is isomorphic to either $\mathbb{Z}$ or $\mathbb{Z}^2$. Now since $H$ has finite index in $C_G(h)$ and $C_G(h)$ has finite index in $G$ we get that $H$ has finite index in $G$. 

If $H \cong \mathbb{Z}$ then $G$ is a virtually cyclic torsion-free group, and hence it must be isomorphic to $\mathbb{Z}$. If $H \cong \mathbb{Z}^2$ then $G$ is a torsion-free group that contains $\mathbb{Z}^2$ as a finite-index subgroup. Such $G$ must be isomorphic to either $\mathbb{Z}^2$ or to the fundamental group of the Klein bottle. Since the latter contains elements that are not virtually central, we conclude that $G \cong \mathbb{Z}^2$.
\end{proof}


\section{Good geodesics}\label{sec:constructgood}
The proofs in Section~\ref{sec:fixedpoints} require going through the construction of good geodesics (unlike proofs in previous sections, where only certain `formal' properties were needed). In this section we give a sketch of this construction.

In order to define good geodesics we first describe the construction of \emph{Euclidean ge\-o\-des\-ics}. This construction is fairly involved, and hence it is divided into a few steps. Our exposition is based on \cite[Sections 7--9]{OP} (we refer the reader there for the proofs of various statements discussed below). Throughout this section let $X$ be a systolic complex. Some notions appearing in the construction are presented in Figures~\ref{fig:genericposition} and \ref{fig:alpha_alphaprim} in Section~\ref{sec:fixedpoints} (in the special case of $X= \E$).

\subsection{Directed geodesics}\label{subsec:directedgeo}
Let $x, y \in X$ be two vertices and put $n= d(x,y)$. A \emph{directed geodesic from} $x$ \emph{to} $y$ is a sequence of simplices $(\sigma_i)_{i=0}^n$ such that $\sigma_0 =x$, $\sigma_n=y$ and the following two conditions are satisfied:
\begin{enumerate}
\item any two consecutive simplices $\sigma_i$ and $\sigma_{i+1}$ are disjoint and together they span a simplex of $X$,
\item for any three consecutive simplices $\sigma_{i-1}, \sigma_i, \sigma_{i+1}$we have \[\mathrm{Res}(\sigma_{i-1}, X) \cap B_1(\sigma_{i+1},X)= \sigma_i,\]
where $\mathrm{Res}(\sigma_{i-1}, X)$ is the union of all simplices of $X$ that contain $\sigma_{i-1}$.

\end{enumerate}

A directed geodesic from $x$ to $y$ always exists and it is unique. One can show that $\sigma_i \subset S_i(x,X) \cap S_{n-i}(y,X)$, and therefore any sequence of vertices $(v_i)_{i=0}^n$ such that $v_i \in \sigma_i$ is a geodesic. Finally, as the name suggests, directed geodesics in general are not symmetric -- usually a directed geodesic from $x$ to $y$ is not equal to a directed geodesic from $y$ to $x$.

\subsection{Layers}\label{subsec:layers}
The intersection $S_i(x,X) \cap S_{n-i}(y,X)$ is called the \emph{layer} $i$ \emph{between} $x$ \emph{and} $y$ and it is denoted by $L_i$. For any $i$ the layer $L_i$ is convex and $\infty$--large.  (Layers in fact can be defined in the same way for any two convex subcomplexes $V$ and $W$ such that for every $v \in V$ and $w \in W$ one has $d(v,w)= n$ for some fixed $n >0$.)

\begin{convent}\label{conv:directionconvention}
Suppose that $(\sigma_i)_{i=0}^n$ is a directed geodesic from $x$ to $y$ and $(\tau_i)_{i=0}^n$ is a directed geodesic from $y$ to $x$. We introduce the following convention: despite $(\sigma_i)_{i=0}^n$ and $(\tau_i)_{i=0}^n$ go in the opposite directions, we index simplices of $(\tau_i)_{i=0}^n$ in the same direction as for $(\sigma_i)_{i=0}^n$, i.e.,\ $\tau_0=x, \tau_1, \ldots, \tau_{n-1}, \tau_n=y$.
\end{convent}

Observe that both $\sigma_i$ and $\tau_i$ are contained in the layer $L_i$. Define the \emph{thickness of layer} $L_i$ (with respect to $(\sigma_i)_{i=0}^n$ and $(\tau_i)_{i=0}^n$) to be the maximal distance between vertices of $\sigma_i$ and $\tau_i$ (since layers are convex, this distance is always realised inside $L_i$).
The layer is \emph{thin} if its thickness is at most $1$, and it is \emph{thick} otherwise.

A pair of indices $(j,k)$ such that $0 <j<k<n$ and $j< k-1$ is called a \emph{thick interval} if layers $L_j$ and $L_k$ are thin, and for every $i$ such that $j<i<k$ the layer $L_i$ is thick. If for some $i$ we have $j<i<k$ then we say that $i$ belongs to the interval~$(j,k)$.


\subsection{Characteristic surfaces}\label{subsec:charsurfaces}
Let $(j,k)$ be a thick interval and let $s_i \in \sigma_i$ and $t_i \in \tau_i$ be vertices such that for any $j \leqslant i \leqslant k$ the distance between $s_i$ and $t_i$ is equal to the thickness of layer $L_i$. Consider the sequence of vertices \[(s_j, s_{j+1}, \ldots, s_{k-1}, s_{k}, t_k, t_{k-1}, \ldots, t_{j+1}, t_j, s_j).\] Observe that any two consecutive vertices in the above sequence are adjacent and therefore this sequence defines a closed loop which we denote by $\gamma$. In fact $\gamma$ is always an embedded loop (this amounts to saying that $s_j \neq t_j$ and $s_k \neq t_k$). 

Let $S \colon \Delta \to X$ be a minimal surface spanned by $\gamma$, i.e.,\  a simplicial map from a triangulation of a $2$--disk $\Delta$ such that:
\begin{enumerate}
\item the boundary of $\Delta$ is mapped isomorphically to $\gamma$,
\item the disk $\Delta$ consists of the least possible number of triangles (among all disks $\Delta'$ for which there exists a simplicial map $S' \colon \Delta' \to X$ satisfying (1)).
\end{enumerate}
We call $S \colon \Delta \to X$ a \emph{characteristic surface} (for the thick interval $(j,k)$) and we call $\Delta$ a \emph{characteristic disk}. It is a standard fact that a minimal disk is always systolic, i.e.,\ every of its internal vertices is incident to at least $6$ triangles.

The cycle $\gamma$ does not have to be unique, and hence there could be many characteristic surfaces. For any two characteristic surfaces $S \colon \Delta \to X$ and $S' \colon \Delta' \to X$ the disks $\Delta$ and $\Delta'$ are isomorphic. We can thus identify all such disks and denote \emph{the} characteristic disk by $\Delta$. Now for any simplex $\rho \in \Delta$ the images $S(\rho)$ for all possible characteristic surfaces span a simplex of $X$, which we denote $\mathcal{S}(\rho)$. This assignment, called the \emph{characteristic mapping}, respects inclusions, i.e.,\ if $\rho_1 \subseteq \rho_2$ then  $\mathcal{S}(\rho_1) \subseteq \mathcal{S}(\rho_2)$.


\subsection{Geometry of characteristic disks}\label{subsec:geomchardisks}
For any $i$ such that $j\leqslant i \leqslant k$, let $v_i $ and $ w_i$ be vertices of $\Delta$ that are preimages of $s_i $ and $t_i$ respectively, for some characteristic surface $S \colon \Delta \to X$. In fact vertices $v_i$ and $w_i$ are uniquely defined and the sequence \[(v_j, v_{j+1}, \ldots, v_{k-1}, v_{k}, w_k, w_{k-1}, \ldots, w_{j+1}, w_j, v_j)\]
constitutes the boundary of the disk $\Delta$.

Denote by $\E$ the equilateral triangulation of the Euclidean plane. Clearly $\E$ viewed as a simplicial complex is systolic. 
One shows that the characteristic disk $\Delta$ can be isometrically embedded in $\E$ (such disk is called \emph{flat}). Moreover, after embedding $\Delta \subset \E$, the edges $[v_j, w_j]$ and $[v_k, w_k]$ are parallel, and consecutive layers in $\Delta$ between them  are contained in straight lines of $\E$ (treated as subcomplexes of $\E$), that are parallel to the lines containing $[v_j, w_j]$ and $[v_k, w_k]$. In particular for any $i$ the vertices $v_i$ and $w_i$ lie on a straight line inside $\E$. The subpath of this line between $v_i$ and $w_i$ is the unique geodesic between $v_i$ and $w_i$ in $\Delta$, which we denote by $v_iw_i$. The geodesic $v_iw_i$ is in fact equal to the entire layer $i$ in $\Delta$ (between the edges $[v_j, w_j]$ and $[v_k, w_k]$).

Finally, for any characteristic surface $S \colon \Delta \to X$ (and hence for a characteristic mapping $\mathcal{S} \colon \Delta \to X$ as well) the image of the geodesic $v_iw_i$ is contained in the layer $i$ in $X$. Also, any characteristic surface $S \colon \Delta \to X$ restricted to $v_iw_i$ is an isometric embedding.

\subsection{Euclidean diagonals}\label{subsec:euclideandiag}
Given the characteristic disk $\Delta$, for every $i \in \{j, \ldots, k \}$ let $v_i'$ and $w_i'$ be points on the unique geodesic between $v_i$ and $w_i$ in $\Delta$, that are at distance $\frac{1}{2}$ from $v_i$ and $w_i$ respectively. In particular $v'_j=w_j'$ and $v_k'=w_k'$. Consider a piecewise linear loop defined as the concatenation of straight segments between consecutive points in the sequence \[(v'_j=w_j', v_{j+1}', \ldots v_{k-1}', v_k'=w_k', w_{k-1}', \ldots, w_{j+1}', w_j'=v_j'),\]
and let $\Delta'$ be a polygonal domain inside $\Delta$ enclosed by this loop. We call $\Delta'$ the \emph{modified characteristic disk}. We endow $\Delta'$ with a path metric induced from the Euclidean metric on $\E \cong \mathbb{E}^2$. Observe that $\Delta'$ is simply connected, and therefore this path metric is in fact a $\mathrm{CAT}(0)$ metric. (The disk $\Delta'$ does not have to be a convex subset of $\E$ with respect to the Euclidean metric on $\E$, and therefore a $\mathrm{CAT}(0)$ geodesic inside $\Delta'$ does not have to be a straight line.) 

The \emph{Euclidean diagonal} of $\Delta$ is a sequence of simplices $(\rho_i)_{j+1}^{k-1}$ of $\Delta$ defined as follows. Let $\alpha$ be a $\mathrm{CAT}(0)$ geodesic in $\Delta'$ between points $v_j'=w_j'$ and $v_k'=w_k'$. For every $i$ such that $j<i<k$ choose vertices on the geodesic $v_iw_i$, that are closest to the point of intersection $\alpha \cap v_iw_i$.  For any $i$, it is either a single vertex, in that case we put $\rho_i$ to be that vertex, or in the case when $\alpha$ goes through the barycentre of some edge of $v_iw_i$, then we put $\rho_i$ to be this edge. One can show that the Euclidean diagonal for $\Delta$ satisfies the following two conditions:

\begin{enumerate}

	\item for any $i$ such that $j <i< k-1$ simplices $\rho_i$ and $\rho_{i+1}$ span a simplex,
	\item vertices $v_j, w_j, \sigma_{j+1}$ span a simplex and vertices $v_k, w_k, \rho_{k-1}$ span a simplex (in particular $\rho_{j+1}$ and $\rho_{k-1}$ are necessarily vertices).

\end{enumerate}

\subsection{Euclidean geodesics}\label{subsec:euclideangeo}
We are ready now to define Euclidean geodesics.

\begin{de}\label{def:euclideangeo}
The \emph{Euclidean geodesic} between vertices $x$ and $y$ in $X$, such that $d(x,y)=n$ is the sequence of simplices $(\delta_i)_{i=0}^n$ defined as follows.  For any $i$ such that $0<i<n$, if the layer $L_i$ is thin then set \[\delta_i= \mathrm{span}\{\sigma_i, \tau_i\},\] where $\sigma_i$ and $\tau_i$ are the simplices of the directed geodesics between $x$ and $y$ that are contained in layer $L_i$. For any $i$ such that the layer $L_i$ is thick, consider the thick interval $(j,k)$ that contains $i$ and put \[\delta_i=\mathcal{S}(\rho_i),\] where $\rho_i$ is the simplex of Euclidean diagonal that is contained in the layer $i$ in the characteristic disk for $(j,k)$, and $\mathcal{S}$ denotes the characteristic mapping. Finally let  $\delta_0=x$ and $\delta_n=y$.
\end{de} 

By definition, consecutive simplices of the Euclidean geodesic $(\delta_i)_{i=0}^n$  are contained in consecutive layers between $x$ and $y$. Unlike for directed geodesics, not every two consecutive simplices $\delta_i, \delta_{i+1}$ span a simplex of $X$. However, the following holds.

\begin{prop}\cite[Remark 3.1]{OP}\label{prop:euclideanisgeo} Suppose $(\delta_i)_{i=0}^n$ is a Euclidean geodesic between vertices $x$ and $y$. Then there exists a sequence of vertices $(v_i)_{i=0}^n$ such that $v_i \in \delta_i$ and $(v_i)_{i=0}^n$ is a geodesic.
\end{prop}

Also note that Euclidean geodesics are symmetric with respect to their endpoints. We now present two theorems describing the crucial properties of Euclidean geodesics. 
The first one, roughly speaking, says that Euclidean geodesics are coarsely closed under taking subsegments. The second one is a coarse form of a $\mathrm{CAT}(0)$ inequality for Euclidean geodesics.

\begin{tw}\cite[Theorem B]{OP}\label{tw:euclideanisgood} Let $(\delta_i)_{i=0}^n$ be a Euclidean geodesic between vertices $x$ and $y$.
Take $j,k \in \{0, \ldots, n \}$ with $j<k$ and let $(r_i)_{i=j}^k$ be a geodesic such that $r_i \in \delta_i$ for $i \in \{j, \ldots, k \}$. Let $(\delta^{j,k}_i)_{i=j}^k$ denote the Euclidean geodesic between vertices $r_j$ and $r_k$. Then for every $i \in \{j, \ldots,k\}$ for any vertices $v_i \in \delta_i$ and $u_i \in \delta^{j,k}_i$ we have \begin{equation*}d(v_i, u_i) \leqslant C,\end{equation*} where $C>0$ is a universal constant.
\end{tw}

\begin{tw}\cite[Theorem C]{OP}\label{tw:euclideancontracting} 
Let $x, y$ and $\widetilde{y}$ be vertices of $X$ with $d(x,y)= n$ and $d(x,\widetilde{y})=m$. Let $(v_i)_{i=0}^n$ and $(\widetilde{v}_i)_{i=0}^m$ be geodesics such that for all appropriate~$i$~we have $v_i \in \delta_i$ and $\widetilde{v}_i \in \widetilde{\delta}_i$, where $(\delta_i)_{i=0}^n$ and $(\widetilde{\delta}_i)_{i=0}^m$ are the Euclidean geodesics between $x$ and $y$ and between $x$ and $\widetilde{y}$ respectively. Then for any $0 \leqslant c \leqslant 1$ we have \begin{equation*} d(v_{\ent{cn}},\widetilde{v}_{\ent{cm}} ) \leqslant c \cdot d(v_n, \widetilde{v}_m) +C,
\end{equation*}
where $C>0$ is a universal constant.
\end{tw}


\subsection{Good geodesics}\label{subsec:goodgeo}
From now on let $C>0$ be a fixed constant which satisfies the assertions of both Theorem~\ref{tw:euclideanisgood} and Theorem~\ref{tw:euclideancontracting}. In particular this means that $C\geqslant 200$ (\cite[p.\ 2877]{OP}). Having an explicit lower bound will be needed in Section~\ref{sec:fixedpoints}.\medskip

Theorem~\ref{tw:euclideanisgood} presents a model behaviour, which motivates the definition of good ge\-o\-des\-ics.

\begin{de}[$C'$--good geodesic]
Let $(v_i)_{i=0}^n$ be a geodesic in $X$. For $j,k \in \{0, \ldots, n \}$ let $(\delta^{j,k}_i)_{i=j}^k$ denote the Euclidean geodesic between vertices $v_j$ and $v_k$. We say that $(v_i)_{i=0}^n$ is a $C'$--\emph{good geodesic} if for every two vertices $v_j $ and $v_k$, for every $i \in \{j, \ldots,k\}$  for any vertex $u_i \in \delta^{j,k}_i$ we have \begin{equation*}d(v_i, u_i) \leqslant C',\end{equation*} where $C'>0$ is a positive constant.
An infinite geodesic is a $C'$--good geodesic if every of its finite subgeodesics is a $C'$--good geodesic. Observe that for a $C'$--good geodesic any of its subgeodesics is a $C'$--good geodesic as well.
\end{de}

\begin{de}[Good geodesic]\label{def:goodgeo}
A geodesic $(v_i)_i$ (finite or infinite) is a \emph{good geodesic} if it is a $C$--good geodesic.
\end{de}

In particular, by Theorem~\ref{tw:euclideanisgood} any geodesic arising from Proposition~\ref{prop:euclideanisgeo} is a good geodesic. Consequently, any two vertices of $X$ can be joined by a good geodesic (cf.\ \cite[Corollary 3.3]{OP}).\medskip 

We finish this section with the following two remarks.

\begin{rem}By going through the steps of the construction, one observes that the directed, Euclidean, and good geodesics are preserved by simplicial automorphisms of $X$.
\end{rem}

\begin{rem}The main goal of the construction outlined in this section is to establish Theorem~\ref{tw:contracting}. This theorem plays the key role in showing various properties of the boundary in \cite{OP}. Theorem~\ref{tw:contracting} follows easily from Definition~\ref{def:goodgeo}, Proposition~\ref{prop:euclideanisgeo} and Theorem~\ref{tw:euclideancontracting}. In particular, the constant $D$ appearing in Theorem~\ref{tw:contracting} may be taken to be $3C$. 
\end{rem}

\section{Fixed points on the boundary}\label{sec:fixedpoints}

The purpose of this section is to show that every hyperbolic isometry $h$ of a systolic complex $X$ fixes a pair of points on the boundary. These two points, denoted by $h^{+\infty}$ and $h^{-\infty}$ are the canonical fixed points of $h$, in the sense that for any vertex $x\in X$ we have $(h^n \cdot x)_n \to h^{+\infty}$ and $(h^{-n} \cdot x)_n \to h^{-\infty}$ as $n \to \infty$ in $\overline{X}=X \cup \partial X$.

To obtain $h^{+\infty}$ and $h^{-\infty}$ we show that there exists a bi-infinite good geodesic $\gamma$ such that $\gamma$ and $h \cdot \gamma$ are asymptotic. In principal, one could expect a stronger result, namely the existence of an $h$--invariant good geodesic. However, there are examples of systolic complexes where a hyperbolic isometry has no invariant geodesics at all \cite[Example 1.2]{E2}. It is true though, that for every hyperbolic isometry $h$ there is a geodesic $\gamma$, such that $h \cdot \gamma$ and $\gamma$ are Hausdorff $1$--close \cite[Theorem 1.3]{E2}. Unfortunately, in our construction the distance between $h \cdot \gamma$ and $\gamma$ depends on $\trol{h}$.

\begin{tw}\label{tw:fixedpoints}Let $X$ be a systolic complex on which a group $G$ acts geometrically. Let $h \in G$ be a hyperbolic isometry. Then either there exists a bi-infinite good geodesic which is contained in $\disp{K}{h}$ for some $K=K(h)$ and $\disp{K}{h}$ is $h$--cocompact, or there exists an $h$--invariant good geodesic.
\end{tw}

We now state and prove the main result of this section assuming Theorem~\ref{tw:fixedpoints}.

\begin{prop}\label{prop:fixedpointsconvergence}
Let $G$ be a group acting geometrically on a systolic complex $X$ and let $h \in G$ be a hyperbolic isometry. Then:

\begin{enumerate}

		\item there exist points $h^{-\infty}$ and $ h^{+\infty}$ in the boundary $\partial X$ which are fixed by $h$,

			\item for any vertex $x \in X$ we have $(h^n\cdot x)_n \to h^{+\infty}$ and $(h^{-n} \cdot x)_n \to h^{-\infty}$ as $n \to \infty$ in the compactification $\overline{X} = X \cup \partial_O X$, where $O \in X$ is some base vertex.
\end{enumerate}
\end{prop}

Since for any two vertices $O, O' \in X$ there is a homeomorphism between $X \cup \partial_OX$ and $X \cup \partial_{O'}X$ (see Subsection~\ref{subsec:boundaries}), the choice of $O$ does not really matter. In order to simplify the argument we will choose $O$ during the proof.

\begin{proof}
To show that a sequence $(x_n)_{n=0}^{\infty}$ converges to a point $[\xi] \subset \partial_OX$ in $\overline X$, it is enough to find a sequence $(v_n)_{n=0}^{\infty} \subset \xi$, such that $d(v_n, O) \to \infty$ as $n \to \infty$ and such that $d(v_n, x_n)$ is uniformly bounded (see \cite[Definition~4.1]{OP}).

By Theorem~\ref{tw:fixedpoints} there either exists a bi-infinite good geodesic $\gamma \subset \disp{K}{h}$ and $\disp{K}{h}$ is $h$--cocompact, or there exists an $h$--invariant good geodesic $\gamma$. We first give the proof assuming that there exists a good geodesic $\gamma$ which is $h$--invariant.

Choose a vertex $O \in \gamma$ and parametrise vertices of $\gamma$ by integers such that $\gamma(0)= O$ and $h \cdot \gamma(0)= \gamma(L(h))$. Then $\gamma$ splits into two good geodesic rays $\gamma\big|_{[0, +\infty]}$ and $\gamma\big|_{[0,-\infty]}$
starting at $O$. Define $h^{+\infty} = [\gamma\big|_{[0, +\infty]}]$ and  $h^{-\infty} = [\gamma\big|_{[0, -\infty]}]$. 
Since $\gamma$ is $h$--invariant, for every $i \geqslant 0$ we have 
\[d(\gamma\big|_{[0,\pm\infty]}(i), h \cdot \gamma\big|_{[0,\pm\infty]}(i)) =L(h),\] and thus both $h^{+\infty}$ and $h^{-\infty}$ are fixed by $h$. 

Let $x \in X$ be an arbitrary vertex and let $M=d(x,O)$. For any $n \geqslant0$ we have:
\begin{enumerate}
 \item $h^n \cdot O \in \gamma\big|_{[0, +\infty]}$, 
 \item $d(h^n \cdot O,O)  \to \infty$ as $ n \to \infty$,
  \item $d(h^n \cdot x, h^n \cdot O)=M$.
 \end{enumerate}
 This implies that $h^n \cdot x \to [\gamma\big|_{[0, +\infty]}]= h^{+\infty}$ and $h^{-n} \cdot x \to [\gamma\big|_{[0, -\infty]}] = h^{-\infty}$.\medskip

In the case where $\gamma \subset \disp{K}{h}$ is a bi-infinite good geodesic and $\disp{K}{h}$ is $h$--cocompact we proceed similarly as above. Choosing a vertex $O \in \gamma$ splits $\gamma$ into two good geodesic rays. We denote them by $\gamma\big|_{[0, +\infty]}$ and $\gamma\big|_{[0,-\infty]}$, even though we did not specify how we choose an orientation of $\gamma$ (it will become clear from the proof). Consequently, let $h^{+\infty} = [\gamma\big|_{[0, +\infty]}]$ and  $h^{-\infty} = [\gamma\big|_{[0, -\infty]}]$. Both $\gamma\big|_{[0, +\infty]}$ and $\gamma\big|_{[0,-\infty]}$ are contained in $\disp{K}{h}$ and therefore for every $i \geqslant 0 $ we have \[d(\gamma\big|_{[0,\pm\infty]}(i), h \cdot \gamma\big|_{[0,\pm\infty]}(i)) \leqslant K,\] and hence $h$ fixes both $h^{+\infty}$ and $h^{-\infty}$.

Let $x \in X$ be an arbitrary vertex and let $M=d(x,O)$. Since $\disp{K}{h}$ is $h$--cocompact and $\gamma$ is bi-infinite it follows that there exists $R>0$ such that \[\disp{K}{h} \subset B_R(\gamma,X).\] Consider the sequence $(h^n \cdot O)_n$. We do not necessarily have  $(h^n \cdot O)_n \subset \gamma\big|_{[0, \infty]}$, but for any $n \geqslant 0$ there exists a vertex $v_n \in \gamma\big|_{[0, \infty]}$ with $d(h^n \cdot O,v_n) \leqslant R$. Then, by the triangle inequality we have $d(O, v_n) \to \infty$ as $n \to \infty$ since $d(O, h^n \cdot O) \to \infty$ as $n \to \infty$. 
Finally, for any $n \geqslant 0$ we have  \[d( h^n \cdot x, v_n) \leqslant d(h^n \cdot x, h^n \cdot O)+ d(h^n \cdot O, v_n) \leqslant M+R,\] and therefore $h^n \cdot x \to [\gamma\big|_{[0, +\infty]}]= h^{+\infty}$ and $h^{-n} \cdot x \to [\gamma\big|_{[0, -\infty]}] = h^{-\infty}$.\end{proof}

It remains to prove Theorem~\ref{tw:fixedpoints}. Since the proof is fairly long we outline it first. 

\subsubsection*{Outline of the proof of Theorem~\ref{tw:fixedpoints}} Observe that the action of $h$ on $X$ preserves $\minset{h}$. We consider two cases: when $\minset{h}$ is $h$--cocompact and when it is not $h$--cocompact. These two cases will lead to the two respective claims of the theorem.

In the first case we show in Lemma~\ref{lem:euclideanisclosetomin} that for any two vertices $x,y \in \minset{h}$ the Euclidean geodesic between these vertices is contained in $\disp{K}{h}$ for some $K>0$. This is achieved by showing that both directed geodesics between these vertices, and characteristic disks spanned by those geodesics, belong to $\disp{K}{h}$. The main tool in the proof of Lemma~\ref{lem:euclideanisclosetomin} is the Fellow Traveller Property of directed geodesics. Then we construct the desired good geodesic, roughly, as a limit of good geodesics between vertices $h^{-n} \cdot x $ and $ h^n \cdot x$ for a fixed vertex $x\in \minset{h}$.

In the second case, since $C_G(h)$ acts geometrically on $\minset{h}$, we deduce that there is a hyperbolic isometry $g$ that commutes with $h$, such that $\langle g,h \rangle \cong \mathbb{Z}^2$. Thus by the systolic Flat Torus Theorem there is a flat $F \subset X$ on which the subgroup $\langle g,h\rangle$ acts by translations. Then using Lemma~\ref{lem:cat0inflatisgood} we find an $h$--invariant $C'$--good geodesic inside $F$ (treated as a systolic complex on its own), for a certain $C'>0$. By Lemma~\ref{lem:goodinflatisgood} any $C'$--good geodesic in $F$ is a $(C'+10)$--good geodesic in $X$. In the above procedure we are able to choose $C'$ so that $C'+10$ is less than $C$ and therefore the constructed $(C'+10)$--good geodesic is a good geodesic in $X$.

Before giving the proof of Theorem~\ref{tw:fixedpoints} we state and prove the three lemmas mentioned above.

\begin{lem}\label{lem:euclideanisclosetomin} Consider two vertices $x,y \in \minset{h} \subset X$ and let $(\delta_i)_{i=0}^{n}$ be the Euclidean geodesic between $x$ and $y$. Then we have $(\delta_i)_{i=0}^{n} \subset \disp{K}{h}$, where $K=9 \cdot \trol{h}+6$.
\end{lem}
\begin{proof}
Let $(\sigma_i)_{i=0}^{n}$ be a directed geodesic from $x$ to $y$ (i.e.,\ $\sigma_0=x$ and $\sigma_n=y$). Then, by the Fellow Traveller Property [JS06, Proposition 11.2] applied to directed geodesics $(\sigma_i)_{i=0}^{n}$ and $(h \cdot \sigma_i)_{i=0}^{n}$, for each $i\in \{0, \ldots,n\}$ for any vertex $s \in \sigma_i$ we have \[d(s, h \cdot s) \leqslant 3 \cdot \mathrm{max}\{d(x, h\cdot x), d(y, h\cdot y)\} +1= 3\cdot \trol{h}+1, \]
since $x,y \in \minset{h}.$ Put $K'=3\cdot \trol{h}+1$. By the above inequality we get that $(\sigma_i)_{i=0}^{n}\subset \disp{K'}{h}$.

Now let $(\sigma_i)_{i=0}^{n}$ be as above, and let $(\tau_i)_{i=0}^{n}$ be the directed geodesic from $y$ to $x$ (see Convention~\ref{conv:directionconvention}). 
Clearly, by the argument above, we also have $(\tau_i)_{i=0}^{n}\subset \disp{K'}{h}$.
If for some $i \in \{0, \ldots, n\}$ the layer $L_i$ is thin, then by definition $\delta_i= \mathrm{span} \{\sigma_i, \tau_i\}$ and therefore it is contained in $\disp{K'}{h}$, since both $\sigma_i$ and $\tau_i$ are so. If the layer $L_i$ is thick then we proceed as follows. 

Take any vertex $z \in \delta_i$. We claim that there exist vertices $s_i \in \sigma_i$ and $t_i \in \tau_i$ and a geodesic $\alpha$ between $s_i$ and $t_i$, such that $z$ lies on $\alpha$. Indeed, consider a thick interval that contains $i$, let $\Delta$ be an appropriate characteristic disk and let $v_iw_i$ be a geodesic in $\Delta$ that is the layer $i$ in $\Delta$. Any characteristic surface $S \colon \Delta \to X$ restricted to $v_iw_i$ is an isometric embedding. Moreover, any vertex of $\delta_i$ lies in the image $S(v_iw_i)$ for some such surface. Take a surface $S \colon \Delta \to X$ such that $z \in S(v_iw_i)$ and put $s_i=S(v_i)$, $t_i=S(w_i)$ and let $\alpha=S(v_iw_i)$. This proves the claim.

Put $m= d(s_i,t_i)$ and let $(\rho_j)_{j=0}^m$ be a directed geodesic from $s_i$ to $t_i$ (actually, here the direction is not important).
Let $(u_j)_{j=0}^m$ be a geodesic such that $u_j \in \rho_j$ for every $j \in \{0, \ldots, m\}$ (in particular $u_0=s_i$ and $u_m=t_i$).

Since the layer $L_i$ is convex, both $\alpha$ and $(u_j)_{u=0}^m$ are contained in $L_i$. Since $L_i$ is $\infty$--large, any two geodesics with the same endpoints are Hausdorff $1$--close \cite[Lemma 2.3]{JS2}.  Therefore $\alpha$ and $(u_j)_{j=0}^m$ are $1$--close. Finally, by the Fellow Traveller Property (applied to $(\rho_j)_{j=0}^m$ and $(h \cdot \rho_j)_{j=0}^m$) for any $j \in \{0, \ldots, m\}$ we have \[d(u_j, h \cdot u_j) \leqslant 3 \cdot \mathrm{max}\{d(s_i, h\cdot s_i), d(t_i, h\cdot t_i)\} +1 \leqslant 3\cdot K'+1, \]
since both $s_i$ and $t_i$ belong to $\disp{K'}{h}$. Because $\alpha$ and $(u_j)_{j=0}^m$ are $1$--close and $z \in \alpha$, the above inequality implies that $d(z, h \cdot z) \leqslant 3 \cdot K'+1 +2$ and hence $z \in \disp{3K'+3}{h}$.
Since $z \in \delta_i$ was arbitrary we obtain that $\delta_i \subset \disp{3K'+3}{h}$ for any $i$ such that $L_i$ is thick.
This, together with the assertion that $\delta_i \subset \disp{K'}{h}$ for any $i$ such that $L_i$ is thin, finishes the proof of the lemma, as $3 \cdot K' +3=9 \cdot \trol{h}+6$.
\end{proof}

In the next lemma we study the equilaterally triangulated Euclidean plane $\E$. We view it simultaneously as a systolic simplicial complex and as a $\mathrm{CAT}(0)$ metric space (cf.\ Subsections~\ref{subsec:geomchardisks} and~\ref{subsec:euclideandiag}). We denote the $\mathrm{CAT}(0)$ distance between points of $\E$ by $d_{\hspace{1pt} \mathbb{E}^2}$ in order to distinguish it from the standard (combinatorial) distance $d$.

\begin{lem}\label{lem:cat0inflatisgood} Let $h$ be a hyperbolic isometry of $\E$ and let $\gamma \subset \E$ be an $h$--invariant geodesic. Suppose that $\beta$ is a $\mathrm{CAT}(0)$ geodesic (i.e., a straight line) such that $\gamma$ and $\beta$ are Hausdorff $K$--close with respect to the $\mathrm{CAT}(0)$ distance, for some $K>0$. Then $\gamma$ is a $(\frac{4K}{\sqrt{3}}+1)$--good geodesic.
\end{lem}

\begin{proof}
We first give the idea of the proof. We observe that for any thick interval in $\E$ there is a unique characteristic surface. After identifying the characteristic disk $\Delta$ with its image, the $\mathrm{CAT}(0)$ geodesic $\alpha$ in the modified characteristic disk $\Delta'$ is uniformly close to $\beta$, and this distance depends only on $K$. This will imply the lemma as the simplices of the Euclidean geodesic are $1$--close to $\alpha$, and $\gamma$ is $K$--close to $\beta$. The case of thin layers will follow easily from the methods used to prove the case of thick intervals.\smallskip

We now begin the proof. Let $x$ and $y$ be any two vertices of $\gamma$. Let $(\sigma_i)_{i=0}^n$ be the directed geodesic from $x$ to $y$ and let $(\tau_i)_{i=0}^n$ be the directed geodesic going in the opposite direction (again, simplices of $(\tau_i)_{i=0}^n$ are indexed in the same direction as for $(\sigma_i)_{i=0}^n$, see Convention~\ref{conv:directionconvention}). One checks that these geodesics have the form shown in Figure~\ref{fig:genericposition}.

\begin{figure}[!h]
\centering
\begin{tikzpicture}[scale=0.7]
\begin{scope}
    \clip(-8,4) rectangle (8.5,-3);

\begin{scope}[color=black, very thin]
\draw (-8,0) to (10,0);
\draw (-8,1) to (10,1);
\draw (-8,2) to (10,2);
\draw (-8,3) to (10,3);
\draw (-8,4) to (10,4);
\draw (-8,5) to (10,5);
\draw (-8,6) to (10,6);
\draw (-8,-1) to (10,-1);
\draw (-8,-2) to (10,-2);
\draw (-8,-3) to (10,-3);
\draw (-8,-4) to (10,-4);
\draw(-8,4) to (-7,6);
\draw(-8,2) to (-6,6);
\draw(-8,0) to (-5,6);
\draw(-8,-2) to (-4,6);
\draw(-8,-4) to (-3,6);
\draw(-7,-4) to (-2,6);
\draw(-6,-4) to (-1,6);
\draw(-5,-4) to (0,6);
\draw(-4,-4) to (1,6);
\draw(-3,-4) to (2,6);
\draw(-2,-4) to (3,6);
\draw(-1,-4) to (4,6);
\draw(0,-4) to (5,6);
\draw(1,-4) to (6,6);
\draw(2,-4) to (7,6);
\draw(3,-4) to (8,6);
\draw(4,-4) to (9,6);
\draw(5,-4) to (10,6);
\draw(6,-4) to (10,4);
\draw(7,-4) to (10,2);
\draw(8,-4) to (10,0);
\draw(9,-4) to (10,-2);

\begin{scope}[yscale=-1,xscale=1, shift={(0,-2)}]
\draw(-8,4) to (-7,6);
\draw(-8,2) to (-6,6);
\draw(-8,0) to (-5,6);
\draw(-8,-2) to (-4,6);
\draw(-8,-4) to (-3,6);
\draw(-7,-4) to (-2,6);
\draw(-6,-4) to (-1,6);
\draw(-5,-4) to (0,6);
\draw(-4,-4) to (1,6);
\draw(-3,-4) to (2,6);
\draw(-2,-4) to (3,6);
\draw(-1,-4) to (4,6);
\draw(0,-4) to (5,6);
\draw(1,-4) to (6,6);
\draw(2,-4) to (7,6);
\draw(3,-4) to (8,6);
\draw(4,-4) to (9,6);
\draw(5,-4) to (10,6);
\draw(6,-4) to (10,4);
\draw(7,-4) to (10,2);
\draw(8,-4) to (10,0);
\draw(9,-4) to (10,-2);
\end{scope}
\end{scope}

\definecolor{darkyellow}{RGB}{150,130,20}
\draw[line width=3, yellow!50] (1,4)--(4.5,-3);

\node[darkyellow, right]  at (4.25,-2.1) {$l_{2m+1}$};

\end{scope}



\begin{scope}[red, ultra thick]

\draw[very thick, dotted] (-7,-2)-- (0.5,3);

\draw[very thick, dotted] (7.5,3)-- (0,-2);
\draw[very thick, dotted](0.5,3) -- (7.5, 3);
\draw[very thick, dotted](5.5,3.5) -- (7.5, 3);
\draw[very thick, dotted](5.5,2.5) -- (7.5, 3);

\draw[very thick, dotted](0,-2) -- (-7, -2);

\draw[very thick, dotted](-5, -1.5) -- (-7, -2);
\draw[very thick, dotted](-5,-2.5) -- (-7, -2);
\end{scope}

\draw[fill] (-7,-2)   circle [radius=0.1];
\draw[ultra thick] (-6, -2) --(-6.5, -1);

\begin{scope}[shift={(1.5,1)}]
\draw[fill] (-7,-2)   circle [radius=0.1];
\draw[ultra thick] (-6, -2) --(-6.5, -1);
\end{scope}
\begin{scope}[shift={(3,2)}]
\draw[fill] (-7,-2)   circle [radius=0.1];
\draw[ultra thick] (-6, -2) --(-6.5, -1);
\end{scope}
\begin{scope}[shift={(4.5,3)}]
\draw[fill] (-7,-2)   circle [radius=0.1];
\draw[ultra thick] (-6, -2) --(-6.5, -1);
\end{scope}
\begin{scope}[shift={(6,4)}]
\draw[fill] (-7,-2)   circle [radius=0.1];
\draw[ultra thick] (-6, -2) --(-6.5, -1);
\end{scope}

\draw[fill] (0.5,3)   circle [radius=0.1];
\draw[fill] (1.5,3)   circle [radius=0.1];
\draw[fill] (2.5,3)   circle [radius=0.1];
\draw[fill] (3.5,3)   circle [radius=0.1];
\draw[fill] (4.5,3)   circle [radius=0.1];
\draw[fill] (5.5,3)   circle [radius=0.1];
\draw[fill] (6.5,3)   circle [radius=0.1];
\draw[fill] (7.5,3)   circle [radius=0.1];


\node  [above left] at (-7,-2) {$x$};
\node  [above right] at (7.5,3) {$y$};

\node  [ left] at (-6.5,-1) {$\sigma_1$};
\node  [above left] at (-5.5,-1) {$\sigma_2$};
\node  [ left] at (-5,0) {$\sigma_3$};

\node  [below right] at (-0.15,2) {$\sigma_{2m-1}$};
\node  [above] at (0.4,3) {$\sigma_{2m}$};
\node  [above] at (1.6,3) {$\sigma_{2m+1}$};

\node  [above] at (6.55,3) {$\sigma_{n-1}$};
\node  [above] at (5.35,3) {$\sigma_{n-2}$};
\node  [above] at (-3,2) {$ (\sigma_{i})_{i=0}^n   $} ;


\begin{scope}[shift={(0.5,1)}, rotate=180]
\draw[fill] (-7,-2)   circle [radius=0.1];
\draw[ultra thick] (-6, -2) --(-6.5, -1);

\begin{scope}[shift={(1.5,1)}]
\draw[fill] (-7,-2)   circle [radius=0.1];
\draw[ultra thick] (-6, -2) --(-6.5, -1);
\end{scope}
\begin{scope}[shift={(3,2)}]
\draw[fill] (-7,-2)   circle [radius=0.1];
\draw[ultra thick] (-6, -2) --(-6.5, -1);
\end{scope}
\begin{scope}[shift={(4.5,3)}]
\draw[fill] (-7,-2)   circle [radius=0.1];
\draw[ultra thick] (-6, -2) --(-6.5, -1);
\end{scope}
\begin{scope}[shift={(6,4)}]
\draw[fill] (-7,-2)   circle [radius=0.1];
\draw[ultra thick] (-6, -2) --(-6.5, -1);
\end{scope}

\draw[fill] (0.5,3)   circle [radius=0.1];
\draw[fill] (1.5,3)   circle [radius=0.1];
\draw[fill] (2.5,3)   circle [radius=0.1];
\draw[fill] (3.5,3)   circle [radius=0.1];
\draw[fill] (4.5,3)   circle [radius=0.1];
\draw[fill] (5.5,3)   circle [radius=0.1];
\draw[fill] (6.5,3)   circle [radius=0.1];
\draw[fill] (7.5,3)   circle [radius=0.1];
\end{scope}

\node  [below] at (5,0) {$ (\tau_{i})_{i=0}^n   $} ;
\node  [below] at (-6,-2) {$\tau_1$};
\node  [below] at (-5,-2) {$\tau_2$};
\node  [below right] at (-0.25,-2) {$\tau_{n-2m}$};
\node  [ right] at (1,-2) {$\tau_{n-2m+1}$};
\node  [above] at (0.4,3) {$\sigma_{2m}$};
\node  [below right] at (1.35,-1) {$\tau_{n-2m+2}$};
\node  [below] at (7.55,2) {$\tau_{n-1}$};
\node  [below] at (6.35,2) {$\tau_{n-2}$};

\end{tikzpicture}
\caption{Generic form and position of the directed geodesics $ (\sigma_{i})_{i=0}^n$ and $ (\tau_{i})_{i=0}^n$ in $\E$.}
\label{fig:genericposition}
\end{figure}

 For $k \in \{0,1,\ldots, n\}$ let $l_k$ denote the infinite line in $\E$ that contains $\sigma_k$ and~$\tau_k$. In particular, for $k \in \{1,\ldots, n-1\}$ the line $l_k$ contains the layer $L_k$ between $x$~and~$y$. Note that the geodesic $(\sigma_i)_{i=0}^n$ splits into two parts: the part $(\sigma_i)_{i=0}^{2m}$, where $\sigma_i$ is a vertex for even $i$ and an edge for odd $i$, and the part $(\sigma_i)_{i=2m+1}^n$, which consists entirely of vertices. Now observe that if $2m \in \{0,2, n-1, n\}$ then for every $k \in \{1, \ldots,n-1 \}$ the layer $L_k$ with respect to $(\sigma_i)_{i=0}^n$ and $(\tau_i)_{i=0}^n$  is thin. If $2m=n-2$ then for $k \in \{1, 2m+1\}$ the layer $L_k$ is thin, and for $k \in \{ 2, \ldots, 2m \}$ the layer $L_k$ is thin for $k$ even, and thick of thickness $2$ for $k$ odd. The above five cases ($2m \in \{0, 2, n-2, n-1, n\}$) will be dealt with at the end.\smallskip

 Now assume that $2< 2m <n-2$. In this case for $k\in \{1,2, n-1,n-2\}$ layers $L_k$ are thin and $(2, n-2)$ is a thick interval (i.e.,\ layers $L_3, L_4, \ldots, L_{n-3}$ are thick). Let $S \colon \Delta \to \E$ be a characteristic surface for the interval $(2,n-2)$. The image of $S$ is presented in Figure~\ref{fig:alpha_alphaprim}. Observe that $S$ is the unique characteristic surface for this interval, and that it is an isometric embedding. Therefore we can identify $\Delta$ with $S(\Delta)$ and treat it as a subcomplex of $\E$.

Let $\alpha$ denote the $\mathrm{CAT}(0)$ geodesic in the modified characteristic disc $\Delta' \subset \Delta \subset \E$ between the midpoints of edges $[\sigma_2, \tau_2]$ and $[\sigma_{n-2}, \tau_{n-2}]$. Denote these midpoints by $v_2$ and $v_{n-2}$ respectively. Let $\alpha'$ be a $\mathrm{CAT}(0)$ geodesic joining $x$ and $y$. Both $\alpha$ and $\alpha'$ are shown in Figure~\ref{fig:alpha_alphaprim}.  Note that $\alpha'$  does not appear in the construction of the Euclidean diagonal; it is an auxiliary geodesic that will be used to estimate distances between $\alpha$ and $\beta$.

\begin{figure}[!h]
\centering
\begin{tikzpicture}[scale=0.7]
\definecolor{vlgray}{RGB}{225,225,225}

\colorlet{backdark}{red!15}
\colorlet{backlight}{red!10}

\begin{scope}
    \clip(-8,4) rectangle (8.5,-3);


\draw[fill=backdark] (-5,-2) -- (-5.5, -1)--(-5,0)--(-4,0)--(-3.5, 1)--(-2.5,1)--(-2,2)--
(-1,2)--(-0.5,3)--(5.5,3)--(6,2)--(5.5,1)--(4.5,1)--(4,0)--(3,0)--(2.5,-1)--(1.5,-1)--(1,-2)--(-5,-2);

\draw[fill=backlight] (-5.25,-1.5 )--(0.75, -1.5) -- (1.25,-0.5)--(2.25,-0.5)--(2.75,0.5)--(3.75, 0.5)--(4.25, 1.5)--(5.25, 1.5)  -- (5.75,2.5)--
(-0.25, 2.5)--(-0.75, 1.5)--(-1.75, 1.5)--(-2.25, 0.5)--(-3.25, 0.5)--(-3.75, -0.5)--(-4.75, -0.5)--(-5.25, -1.5);

\begin{scope}[color=black, very thin]
\draw (-8,0) to (10,0);
\draw (-8,1) to (10,1);
\draw (-8,2) to (10,2);
\draw (-8,3) to (10,3);
\draw (-8,4) to (10,4);
\draw (-8,5) to (10,5);
\draw (-8,6) to (10,6);
\draw (-8,-1) to (10,-1);
\draw (-8,-2) to (10,-2);
\draw (-8,-3) to (10,-3);
\draw (-8,-4) to (10,-4);
\draw(-8,4) to (-7,6);
\draw(-8,2) to (-6,6);
\draw(-8,0) to (-5,6);
\draw(-8,-2) to (-4,6);
\draw(-8,-4) to (-3,6);
\draw(-7,-4) to (-2,6);
\draw(-6,-4) to (-1,6);
\draw(-5,-4) to (0,6);
\draw(-4,-4) to (1,6);
\draw(-3,-4) to (2,6);
\draw(-2,-4) to (3,6);
\draw(-1,-4) to (4,6);
\draw(0,-4) to (5,6);
\draw(1,-4) to (6,6);
\draw(2,-4) to (7,6);
\draw(3,-4) to (8,6);
\draw(4,-4) to (9,6);
\draw(5,-4) to (10,6);
\draw(6,-4) to (10,4);
\draw(7,-4) to (10,2);
\draw(8,-4) to (10,0);
\draw(9,-4) to (10,-2);

\begin{scope}[yscale=-1,xscale=1, shift={(0,-2)}]
\draw(-8,4) to (-7,6);
\draw(-8,2) to (-6,6);
\draw(-8,0) to (-5,6);
\draw(-8,-2) to (-4,6);
\draw(-8,-4) to (-3,6);
\draw(-7,-4) to (-2,6);
\draw(-6,-4) to (-1,6);
\draw(-5,-4) to (0,6);
\draw(-4,-4) to (1,6);
\draw(-3,-4) to (2,6);
\draw(-2,-4) to (3,6);
\draw(-1,-4) to (4,6);
\draw(0,-4) to (5,6);
\draw(1,-4) to (6,6);
\draw(2,-4) to (7,6);
\draw(3,-4) to (8,6);
\draw(4,-4) to (9,6);
\draw(5,-4) to (10,6);
\draw(6,-4) to (10,4);
\draw(7,-4) to (10,2);
\draw(8,-4) to (10,0);
\draw(9,-4) to (10,-2);
\end{scope}
\end{scope}

\end{scope}

\definecolor{darkyellow}{RGB}{150,130,20}
\draw[line width=3, yellow!50] (1,4)--(4.5,-3);

\node[darkyellow, right]  at (4.25,-2.1) {$l_k$};


\draw[very thick, dashed] (-5,-2) -- (-5.5, -1)--(-5,0)--(-4,0)--(-3.5, 1)--(-2.5,1)--(-2,2)--
(-1,2)--(-0.5,3)--(5.5,3)--(6,2)--(5.5,1)--(4.5,1)--(4,0)--(3,0)--(2.5,-1)--(1.5,-1)--(1,-2)--(-5,-2);

\draw[ultra thick] (-6, -2) --(-6.5, -1);

\begin{scope}[shift={(1.5,1)}]
\draw[fill] (-7,-2)   circle [radius=0.1];
\draw[ultra thick] (-6, -2) --(-6.5, -1);
\end{scope}
\begin{scope}[shift={(3,2)}]
\draw[fill] (-7,-2)   circle [radius=0.1];
\draw[ultra thick] (-6, -2) --(-6.5, -1);
\end{scope}
\begin{scope}[shift={(4.5,3)}]
\draw[fill] (-7,-2)   circle [radius=0.1];
\draw[ultra thick] (-6, -2) --(-6.5, -1);
\end{scope}
\begin{scope}[shift={(6,4)}]
\draw[fill] (-7,-2)   circle [radius=0.1];
\draw[ultra thick] (-6, -2) --(-6.5, -1);
\end{scope}

\draw[fill] (0.5,3)   circle [radius=0.1];
\draw[fill] (1.5,3)   circle [radius=0.1];
\draw[fill] (2.5,3)   circle [radius=0.1];
\draw[fill] (3.5,3)   circle [radius=0.1];
\draw[fill] (4.5,3)   circle [radius=0.1];
\draw[fill] (5.5,3)   circle [radius=0.1];
\draw[fill] (6.5,3)   circle [radius=0.1];
\draw[fill] (7.5,3)   circle [radius=0.1];

\draw (-5.25,-1.5 )--(0.75, -1.5) -- (1.25,-0.5)--(2.25,-0.5)--(2.75,0.5)--(3.75, 0.5)--(4.25, 1.5)--(5.25, 1.5)  -- (5.75,2.5)--
(-0.25, 2.5)--(-0.75, 1.5)--(-1.75, 1.5)--(-2.25, 0.5)--(-3.25, 0.5)--(-3.75, -0.5)--(-4.75, -0.5)--(-5.25, -1.5);


\draw[thick, teal] (-7,-2) --(7.5,3);


\node  [above right] at (7.5,3) {$y$};

\node  [ left] at (-6.5,-1) {$\sigma_1$};
\node  [above left] at (-5.5,-1) {$\sigma_2$};
\node  [ left] at (-5,0) {$\sigma_3$};

\node  [below right] at (-0.15,2) {$\sigma_{2m-1}$};
\node  [above] at (0.4,3) {$\sigma_{2m}$};
\node  [above] at (1.6,3) {$\sigma_{k}$};
\node  [darkgray, above] at (3,2.95) { $(=\sigma_{2m+1})$};

\node  [above] at (6.5,3) {$\sigma_{n-1}$};
\node  [above] at (-3,2) {$ (\sigma_{i})_{i=0}^n   $} ;
\node  [above] at (5.35,3) {$\sigma_{n-2}$};


\begin{scope}[shift={(0.5,1)}, rotate=180]
\draw[fill] (-7,-2)   circle [radius=0.1];
\draw[ultra thick] (-6, -2) --(-6.5, -1);

\begin{scope}[shift={(1.5,1)}]
\draw[fill] (-7,-2)   circle [radius=0.1];
\draw[ultra thick] (-6, -2) --(-6.5, -1);
\end{scope}
\begin{scope}[shift={(3,2)}]
\draw[fill] (-7,-2)   circle [radius=0.1];
\draw[ultra thick] (-6, -2) --(-6.5, -1);
\end{scope}
\begin{scope}[shift={(4.5,3)}]
\draw[fill] (-7,-2)   circle [radius=0.1];
\draw[ultra thick] (-6, -2) --(-6.5, -1);
\end{scope}
\begin{scope}[shift={(6,4)}]
\draw[fill] (-7,-2)   circle [radius=0.1];
\draw[ultra thick] (-6, -2) --(-6.5, -1);
\end{scope}

\draw[fill] (0.5,3)   circle [radius=0.1];
\draw[fill] (1.5,3)   circle [radius=0.1];
\draw[fill] (2.5,3)   circle [radius=0.1];
\draw[fill] (3.5,3)   circle [radius=0.1];
\draw[fill] (4.5,3)   circle [radius=0.1];
\draw[fill] (5.5,3)   circle [radius=0.1];
\draw[fill] (6.5,3)   circle [radius=0.1];
\draw[fill] (7.5,3)   circle [radius=0.1];
\end{scope}


\node  [below] at (5,0) {$ (\tau_{i})_{i=0}^n   $} ;
\node  [below] at (-6,-2) {$\tau_1$};
\node  [below] at (-5,-2) {$\tau_2$};
\node  [below right] at (-0.25,-2) {$\tau_{n-2m}$};
\node  [right] at (1,-2) {$\tau_{n-2m+1}$};
\node  [above] at (0.4,3) {$\sigma_{2m}$};
\node  [below right] at (1.35,-1) {$\tau_{n-2m+2}$};
\node  [below] at (7.5,2) {$\tau_{n-1}$};

\node  [below] at (6.35,2) {$\tau_{n-2}$};



\node  [left, orange] at (-7,-2) {$x$};
\draw [fill, orange] (-7,-2)   circle [radius=0.1];

\draw[ultra thick, orange, ->] (-7,-2)--(-6,-2);
\draw[ultra thick, orange, ->] (-7,-2)--(-6.5,-1);
\node  [orange, below] at (-6.6,-2.05) {$\overrightarrow{x \tau_1}$};
\node  [orange, left] at (-6.75,-1.45) {$\overrightarrow{x w}$};

\draw[thick, red] (-5.25,-1.5 )-- (5.75,2.5);

\draw [fill, red] (-5.25, -1.5)   circle [radius=0.075];
\draw [fill, red] (5.75,2.5)      circle [radius=0.075];

\node  [red, below left] at (-5.15, -1.5) {$v_2$};
\node  [red, left ] at (5.75,2.6)   {$v_{n-2}$};

\node  [red] at (-3.35,-1.15)   {$\alpha$};


\node  [teal] at (-4.1,-0.7)   {$\alpha'$};

\draw [fill, teal] (-5.29,-1.41)  circle [radius=0.075];
\draw [fill, teal] (5.79,2.41)      circle [radius=0.075];

\node  [left, teal] at (-5.29,-1.3)  {$q_2$};
\node  [right, teal] at (5.79,2.3)    {$q_{n-2}$};


\draw [fill, red] (0.73* 2.5 +  0.27*2, 0.73+ 0.27*2)  circle [radius=0.075];
\draw [fill, teal] (0.77* 2.5 +  0.23*2, 0.77+ 0.23*2)  circle [radius=0.075];

\node  [left, red] at (0.73* 2.5 +  0.27*2, 0.73+ 0.27*2+0.15)  {$v_k$};
\node  [right, teal] at (0.77* 2.5 +  0.23*2, 0.77+ 0.23*2-0.15)  {$q_k$};

\draw [thick, ->] (-5.5,3)--(-3,0.75);

\node [above ] at (-5.8 ,3) {$S(\Delta)$};

\draw [thick, ->] (-6,2)-- (-2.5,0.25);

\node [above] at (-6.25,1.9) {$\Delta'$};
\end{tikzpicture}
\caption{Image of the characteristic surface $S \colon \Delta \to \E$ and the modified characteristic disc $\Delta'$. Geodesics $\alpha$ and $\alpha'$ and the coordinate system.~The $\mathrm{CAT}(0)$ distance between $v_k$ and $q_k$ is at most $\frac{1}{2}$.}
\label{fig:alpha_alphaprim}
\end{figure}
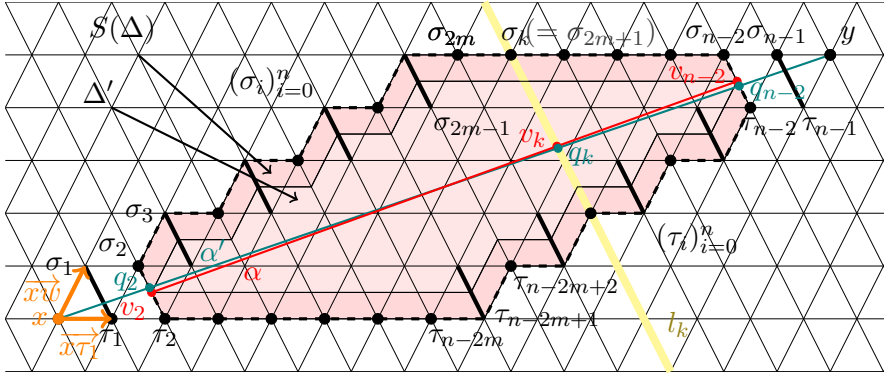

Introduce a coordinate system on $\E$ by setting $x$ as the base vertex, and vectors $\overrightarrow{x \tau_1}$ and $\overrightarrow{x w}$, where $w$ is the vertex of $\sigma_1$ that does not belong to $\tau_1$, as the base vectors (see Figure~\ref{fig:alpha_alphaprim}). Note that both $\alpha$ and $\alpha'$ are contained in the sector of $\E$ bounded by $x$ and half-lines emanating from $x$ in the directions of $\overrightarrow{x \tau_1}$ and $\overrightarrow{x w}$. Moreover, the Euclidean angle between  $\alpha$ and $\overrightarrow{x \tau_1}$ and the Euclidean angle between $\alpha'$ and $\overrightarrow{x \tau_1}$ are both between $0^{\circ}$ and $30^{\circ}$. In particular $\alpha'$ intersects the interior of edges $[\sigma_2, \tau_2]$ and $[\sigma_{n-2}, \tau_{n-2}]$. Call the intersection points $q_2$ and $q_{n-2}$ respectively.

Note that the $\mathrm{CAT}(0)$ distances between $v_2$ and $q_2$ and between $v_{n-2}$ and $q_{n-2}$ are less than $\frac{1}{2}$. For $k \in \{2, \ldots, n-2 \}$ let $q_k$ and $v_k$ denote the intersection points of $\alpha$ and $\alpha'$ respectively with the layer $L_k$ (see Figure~\ref{fig:alpha_alphaprim}). Since the edges $[\sigma_2, \tau_2]$ and $[\sigma_{n-2}, \tau_{n-2}]$ are parallel to all the layers, by elementary Euclidean geometry, for any $k \in \{2,\ldots, n-2\}$ we have:
\begin{equation}\label{eq:flat1} d_{\hspace{1pt} \mathbb{E}^2}(q_k, v_k) \leqslant \mathrm{max} \{ d_{\hspace{1pt} \mathbb{E}^2}(q_2, v_2), d_{\hspace{1pt} \mathbb{E}^2}(q_{n-2}, v_{n-2})\} \leqslant \frac{1}{2}.
\end{equation}

Now consider the $\mathrm{CAT}(0)$ geodesic $\beta$. First we determine the position of $\beta$ with respect to the base vectors $\overrightarrow{x \tau_1}$ and  $\overrightarrow{x w}$. Since both $\overrightarrow{x w}$ and $\overrightarrow{x \tau_1}$ coordinates of $y$ are greater than the coordinates of $x$, and since $\gamma$ is $h$--invariant, it follows that there are vertices $z\in \gamma$ such that $(x,y, z)$ lie on $\gamma$ in this order and both coordinates of $z$ are arbitrarily large. Since $\beta$ stays $K$--close to $\gamma$, it follows that the Euclidean angle directed counter-clockwise from $\overrightarrow{x \tau_1}$ to $\beta$ is between $0^{\circ}$ and $60^{\circ}$. Examples of possible geodesics $\beta$ and $\gamma$ are shown in Figure~\ref{fig:beta_distances}.

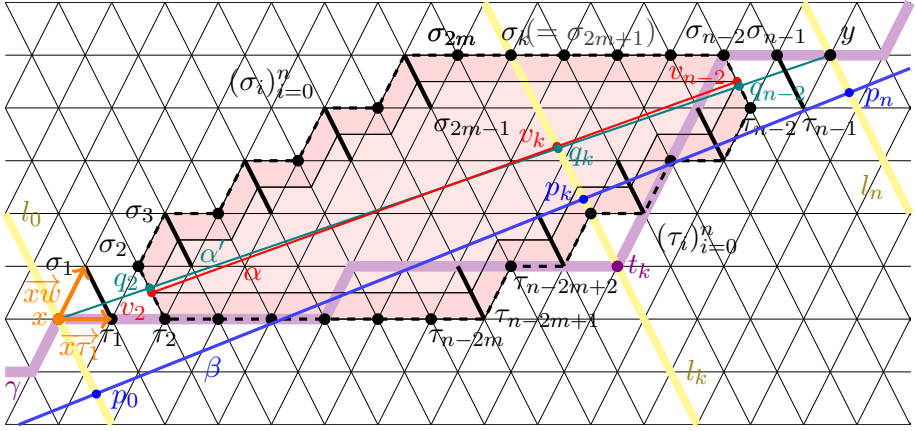
\begin{figure}[!h]
\centering
\begin{tikzpicture}[scale=0.7]
\definecolor{vlgray}{RGB}{225,225,225}

\colorlet{backdark}{red!15}
\colorlet{backlight}{red!10}

\begin{scope}
    \clip(-8,4) rectangle (9,-4);


\draw[fill=backdark] (-5,-2) -- (-5.5, -1)--(-5,0)--(-4,0)--(-3.5, 1)--(-2.5,1)--(-2,2)--
(-1,2)--(-0.5,3)--(5.5,3)--(6,2)--(5.5,1)--(4.5,1)--(4,0)--(3,0)--(2.5,-1)--(1.5,-1)--(1,-2)--(-5,-2);

\draw[fill=backlight] (-5.25,-1.5 )--(0.75, -1.5) -- (1.25,-0.5)--(2.25,-0.5)--(2.75,0.5)--(3.75, 0.5)--(4.25, 1.5)--(5.25, 1.5)  -- (5.75,2.5)--
(-0.25, 2.5)--(-0.75, 1.5)--(-1.75, 1.5)--(-2.25, 0.5)--(-3.25, 0.5)--(-3.75, -0.5)--(-4.75, -0.5)--(-5.25, -1.5);

\begin{scope}[color=black, very thin]
\draw (-8,0) to (10,0);
\draw (-8,1) to (10,1);
\draw (-8,2) to (10,2);
\draw (-8,3) to (10,3);
\draw (-8,4) to (10,4);
\draw (-8,5) to (10,5);
\draw (-8,6) to (10,6);
\draw (-8,-1) to (10,-1);
\draw (-8,-2) to (10,-2);
\draw (-8,-3) to (10,-3);
\draw (-8,-4) to (10,-4);
\draw(-8,4) to (-7,6);
\draw(-8,2) to (-6,6);
\draw(-8,0) to (-5,6);
\draw(-8,-2) to (-4,6);
\draw(-8,-4) to (-3,6);
\draw(-7,-4) to (-2,6);
\draw(-6,-4) to (-1,6);
\draw(-5,-4) to (0,6);
\draw(-4,-4) to (1,6);
\draw(-3,-4) to (2,6);
\draw(-2,-4) to (3,6);
\draw(-1,-4) to (4,6);
\draw(0,-4) to (5,6);
\draw(1,-4) to (6,6);
\draw(2,-4) to (7,6);
\draw(3,-4) to (8,6);
\draw(4,-4) to (9,6);
\draw(5,-4) to (10,6);
\draw(6,-4) to (10,4);
\draw(7,-4) to (10,2);
\draw(8,-4) to (10,0);
\draw(9,-4) to (10,-2);

\begin{scope}[yscale=-1,xscale=1, shift={(0,-2)}]
\draw(-8,4) to (-7,6);
\draw(-8,2) to (-6,6);
\draw(-8,0) to (-5,6);
\draw(-8,-2) to (-4,6);
\draw(-8,-4) to (-3,6);
\draw(-7,-4) to (-2,6);
\draw(-6,-4) to (-1,6);
\draw(-5,-4) to (0,6);
\draw(-4,-4) to (1,6);
\draw(-3,-4) to (2,6);
\draw(-2,-4) to (3,6);
\draw(-1,-4) to (4,6);
\draw(0,-4) to (5,6);
\draw(1,-4) to (6,6);
\draw(2,-4) to (7,6);
\draw(3,-4) to (8,6);
\draw(4,-4) to (9,6);
\draw(5,-4) to (10,6);
\draw(6,-4) to (10,4);
\draw(7,-4) to (10,2);
\draw(8,-4) to (10,0);
\draw(9,-4) to (10,-2);
\end{scope}
\end{scope}

\end{scope}


\draw[line width=4, violet!35] (-8,-3)-- (-7.5,-3)--(-7,-2)--(-2,-2)--(-1.5,-1)--(3.5,-1)--(5.5,3)--(7.5,3)--(8.5,3)--(9,4);

\definecolor{darkyellow}{RGB}{150,130,20}
\draw[line width=3, yellow!50] (1,4)--(5,-4);

\node[darkyellow]  at (5,-3) {$l_k$};

\draw[line width=3, yellow!50] (-8,0)--(-6,-4);

\node[darkyellow]  at (-7.5,0) {$l_0$};

\draw[line width=3, yellow!50] (7,4)--(9,0);

\node[darkyellow]  at (8.3,0.5) {$l_n$};


\draw[very thick, dashed] (-5,-2) -- (-5.5, -1)--(-5,0)--(-4,0)--(-3.5, 1)--(-2.5,1)--(-2,2)--
(-1,2)--(-0.5,3)--(5.5,3)--(6,2)--(5.5,1)--(4.5,1)--(4,0)--(3,0)--(2.5,-1)--(1.5,-1)--(1,-2)--(-5,-2); 

\draw[ultra thick] (-6, -2) --(-6.5, -1);

\begin{scope}[shift={(1.5,1)}]
\draw[fill] (-7,-2)   circle [radius=0.1];
\draw[ultra thick] (-6, -2) --(-6.5, -1);
\end{scope}
\begin{scope}[shift={(3,2)}]
\draw[fill] (-7,-2)   circle [radius=0.1];
\draw[ultra thick] (-6, -2) --(-6.5, -1);
\end{scope}
\begin{scope}[shift={(4.5,3)}]
\draw[fill] (-7,-2)   circle [radius=0.1];
\draw[ultra thick] (-6, -2) --(-6.5, -1);
\end{scope}
\begin{scope}[shift={(6,4)}]
\draw[fill] (-7,-2)   circle [radius=0.1];
\draw[ultra thick] (-6, -2) --(-6.5, -1);
\end{scope}

\draw[fill] (0.5,3)   circle [radius=0.1];
\draw[fill] (1.5,3)   circle [radius=0.1];
\draw[fill] (2.5,3)   circle [radius=0.1];
\draw[fill] (3.5,3)   circle [radius=0.1];
\draw[fill] (4.5,3)   circle [radius=0.1];
\draw[fill] (5.5,3)   circle [radius=0.1];
\draw[fill] (6.5,3)   circle [radius=0.1];
\draw[fill] (7.5,3)   circle [radius=0.1];

\draw (-5.25,-1.5 )--(0.75, -1.5) -- (1.25,-0.5)--(2.25,-0.5)--(2.75,0.5)--(3.75, 0.5)--(4.25, 1.5)--(5.25, 1.5)  -- (5.75,2.5)--
(-0.25, 2.5)--(-0.75, 1.5)--(-1.75, 1.5)--(-2.25, 0.5)--(-3.25, 0.5)--(-3.75, -0.5)--(-4.75, -0.5)--(-5.25, -1.5);


\draw[thick, teal] (-7,-2) --(7.5,3);


\node  [above right] at (7.5,3) {$y$};

\node  [ left] at (-6.5,-1) {$\sigma_1$};
\node  [above left] at (-5.5,-1) {$\sigma_2$};
\node  [ left] at (-5,0) {$\sigma_3$};

\node  [below right] at (-0.15,2) {$\sigma_{2m-1}$};
\node  [above] at (0.4,3) {$\sigma_{2m}$};
\node  [above] at (1.6,3) {$\sigma_{k}$};
\node  [darkgray, above] at (3,2.95) { $(=\sigma_{2m+1})$};
\node  [above] at (6.5,3) {$\sigma_{n-1}$};
\node  [above] at (-3,2) {$ (\sigma_{i})_{i=0}^n   $} ;
\node  [above] at (5.35,3) {$\sigma_{n-2}$};


\begin{scope}[shift={(0.5,1)}, rotate=180]
\draw[fill] (-7,-2)   circle [radius=0.1];
\draw[ultra thick] (-6, -2) --(-6.5, -1);

\begin{scope}[shift={(1.5,1)}]
\draw[fill] (-7,-2)   circle [radius=0.1];
\draw[ultra thick] (-6, -2) --(-6.5, -1);
\end{scope}
\begin{scope}[shift={(3,2)}]
\draw[fill] (-7,-2)   circle [radius=0.1];
\draw[ultra thick] (-6, -2) --(-6.5, -1);
\end{scope}
\begin{scope}[shift={(4.5,3)}]
\draw[fill] (-7,-2)   circle [radius=0.1];
\draw[ultra thick] (-6, -2) --(-6.5, -1);
\end{scope}
\begin{scope}[shift={(6,4)}]
\draw[fill] (-7,-2)   circle [radius=0.1];
\draw[ultra thick] (-6, -2) --(-6.5, -1);
\end{scope}

\draw[fill] (0.5,3)   circle [radius=0.1];
\draw[fill] (1.5,3)   circle [radius=0.1];
\draw[fill] (2.5,3)   circle [radius=0.1];
\draw[fill] (3.5,3)   circle [radius=0.1];
\draw[fill] (4.5,3)   circle [radius=0.1];
\draw[fill] (5.5,3)   circle [radius=0.1];
\draw[fill] (6.5,3)   circle [radius=0.1];
\draw[fill] (7.5,3)   circle [radius=0.1];
\end{scope}


\node  [below] at (5,0) {$ (\tau_{i})_{i=0}^n   $} ;
\node  [below] at (-6,-2) {$\tau_1$};
\node  [below] at (-5,-2) {$\tau_2$};
\node  [below right] at (-0.25,-2) {$\tau_{n-2m}$};
\node  [right] at (1,-2) {$\tau_{n-2m+1}$};
\node  [above] at (0.4,3) {$\sigma_{2m}$};
\node  [below right] at (1.35,-1) {$\tau_{n-2m+2}$};
\node  [below] at (7.5,2) {$\tau_{n-1}$};

\node  [below] at (6.35,2) {$\tau_{n-2}$};



\node  [left, orange] at (-7,-2) {$x$};
\draw [fill, orange] (-7,-2)   circle [radius=0.1];

\draw[ultra thick, orange, ->] (-7,-2)--(-6,-2);
\draw[ultra thick, orange, ->] (-7,-2)--(-6.5,-1);
\node  [orange, below] at (-6.6,-2.05) {$\overrightarrow{x \tau_1}$};
\node  [orange, left] at (-6.75,-1.45) {$\overrightarrow{x w}$};

\draw[thick, red] (-5.25,-1.5 )-- (5.75,2.5);

\draw [fill, red] (-5.25, -1.5)   circle [radius=0.075];
\draw [fill, red] (5.75,2.5)      circle [radius=0.075];

\node  [red, below left] at (-5.15, -1.5) {$v_2$};
\node  [red, left ] at (5.75,2.6)   {$v_{n-2}$};

\node  [red] at (-3.35,-1.15)   {$\alpha$};

\node  [teal] at (-4.1,-0.7)   {$\alpha'$};

\draw [fill, teal] (-5.29,-1.41)  circle [radius=0.075];
\draw [fill, teal] (5.79,2.41)      circle [radius=0.075];

\node  [left, teal] at (-5.29,-1.3)  {$q_2$};
\node  [right, teal] at (5.79,2.3)    {$q_{n-2}$};


\draw [fill, red] (0.73* 2.5 +  0.27*2, 0.73+ 0.27*2)  circle [radius=0.075];
\draw [fill, teal] (0.77* 2.5 +  0.23*2, 0.77+ 0.23*2)  circle [radius=0.075];

\node  [left, red] at (0.73* 2.5 +  0.27*2, 0.73+ 0.27*2+0.15)  {$v_k$};
\node  [right, teal] at (0.77* 2.5 +  0.23*2, 0.77+ 0.23*2-0.15)  {$q_k$};

\draw[very thick, blue!75] (-7.75,-4) --(9,2.75);
\draw [fill, blue] (0.73* 2.5 +  0.27*2+0.5, 0.73+ 0.27*2-1)  circle [radius=0.075];

\node  [left, blue] at (0.73* 2.5 +  0.27*2+0.5, 0.73+ 0.27*2-1+0.15)  {$p_k$};

\draw [fill, blue] (-6.29,-3.42) circle [radius=0.075];
\node[right, blue] at  (-6.29+0.1,-3.42-0.1) {$p_0$};

\draw [fill, blue] (0.29*7.5 + 0.71*8 ,  0.29*3 + 0.71*2) circle [radius=0.075];
\node[right, blue] at (0.29*7.5 + 0.71*8+0.1 ,  0.29*3 + 0.71*2-0.1) {$p_n$};


\draw [fill, violet] (3.5, -1)  circle [radius=0.1];
\node  [right, violet] at (3.5, -1)  {$t_k$};

\node  [right, violet] at (-8.2, -3.35)  {$\gamma$};

\node  [blue] at (-4.1, -2.95)  {$\beta$};
\end{tikzpicture}
\caption{Examples of (parts of) geodesics $\beta$ and $\gamma$ and their position with respect to $\alpha$ and $\alpha'$. Distances between vertices $v_k, q_k, p_k ,t_k$ are measured on the line $l_k$.}
\label{fig:beta_distances}
\end{figure} 

It follows that $\beta$ intersects every line $l_k$ at a single point, which we denote by $p_k$. For any $k \in \{1,\ldots, n-1\}$ let $t_k$ be a vertex of $\gamma$ that belongs to layer $L_k$ (and hence to the line $l_k$), and let $t_0=x$ and $t_n=y$. By assumption, geodesic $\beta$ passes within the $\mathrm{CAT}(0)$ distance $K$ from any $t_k$. Then, since the directed angle from $\overrightarrow{x \tau_1}$ to $\beta$ is between $0^{\circ}$ and $60^{\circ}$, we claim that for any $k \in \{0,\ldots, n\}$ we have: \begin{equation}\label{eq:flat2} d_{\hspace{1pt} \mathbb{E}^2}(p_k, t_k) \leqslant \frac{2K}{\sqrt{3}}.
\end{equation}
To obtain the constant $\frac{2K}{\sqrt{3}}$ $(= \frac{K}{\mathrm{cos}30^{\circ}})$ one checks the extreme cases where $\beta$ is parallel to either $\overrightarrow{x \tau_1}$ or $\overrightarrow{x w}$. In particular we have $d_{\hspace{1pt} \mathbb{E}^2}(p_0, x) \leqslant \frac{2K}{\sqrt{3}}$ and $d_{\hspace{1pt} \mathbb{E}^2}(p_n, y) \leqslant \frac{2K}{\sqrt{3}}$. Now since $\alpha'$ connects $x$ and $y$ and since $\beta$ passes through $p_0$ and $p_n$ it follows that for $k \in \{0, \ldots, n\}$ we have: 
\begin{equation}\label{eq:flat3} d_{\hspace{1pt} \mathbb{E}^2}(p_k, q_k) \leqslant \frac{2K}{\sqrt{3}}.
\end{equation}


Denote by $(\delta_i)_{i=0}^n$ the Euclidean geodesic between $x$ and $y$. By definition, for thin layers we have $\delta_1=\sigma_1$, $\delta_2= [\sigma_2, \tau_2]$, $\delta_{n-2}=[\sigma_{n-2}, \tau_{n-2}]$ and $\delta_{n-1}= \tau_{n-1}$. For thick layers we have $\delta_i=u_i$ where $u_i \in L_i$ is a vertex that is at distance less than $\frac{1}{2}$ from $v_i$ or $\delta_i=[u_i, u'_i]$ if $u_i, u'_i \in L_i$ and $d_{\hspace{1pt} \mathbb{E}^2}(u_i, v_i)= d_{\hspace{1pt} \mathbb{E}^2}(u'_i, q_k)= \frac{1}{2}$. In any case, for $k \in \{3, \ldots,n-3\}$ for any vertex $u_k \in \delta_k$ we have:
\begin{equation}\label{eq:flat4} d_{\hspace{1pt} \mathbb{E}^2}(u_k, v_k) \leqslant \frac{1}{2}.
\end{equation}

Finally, by combining  \eqref{eq:flat1}, \eqref{eq:flat2}, \eqref{eq:flat3} and \eqref{eq:flat4} for any $k \in \{3, \ldots, n-3 \}$ for any $u_k \in \delta_k$ we have:
\begin{equation}\label{eq:duzelk}
\begin{aligned}
d_{\hspace{1pt} \mathbb{E}^2}(t_k, u_k) \leqslant {} & d_{\hspace{1pt} \mathbb{E}^2}(t_k, p_k) +d_{\hspace{1pt} \mathbb{E}^2}(p_k, q_k) +d_{\hspace{1pt} \mathbb{E}^2}(q_k, v_k) +d_{\hspace{1pt} \mathbb{E}^2}(v_k, u_k)\\
					\leqslant {} & \frac{2K}{\sqrt{3}} +\frac{2K}{\sqrt{3}} +\frac{1}{2}+\frac{1}{2}\\
					= {} & \frac{4K}{\sqrt{3}}+1.			
\end{aligned}
\end{equation}
Since all the distances above are measured on the line $l_k$, the same estimate holds for the standard (combinatorial) distance.

For $k \in \{2, n-2\}$ we have $\delta_k= [\sigma_k, \tau_k]$ and one observes that $d_{\hspace{1pt}  \mathbb{E}^2}(q_k, \sigma_k) \leqslant 1$ and $d_{\hspace{1pt} \mathbb{E}^2}(q_k, \tau_k) \leqslant 1$. Combining this with \eqref{eq:flat2} and \eqref{eq:flat3} we obtain the same estimate as in \eqref{eq:duzelk}.  Now for $k \in \{1,n-1\}$, let $q_k$ be the point of intersection $\alpha' \cap \delta_k$. By a direct observation we see that for any vertex $u_k \in \delta_k$ we have $d_{\mathbb{E}^2}(v_k, u_k) \leqslant 1$, and thus again we obtain the same estimate as in \eqref{eq:duzelk}.

We conclude that for any $k \in \{1, \ldots, n-1 \}$ for any $u_k \in \delta_k$ we have 
\[d(t_k, u_k) \leqslant \frac{4K}{\sqrt{3}}+1.\]
Since $x$ and $y$ were arbitrary, this proves that $\gamma$ is a $(\frac{4K}{\sqrt{3}}+1)$--good geodesic.\medskip

For the remaining cases, where geodesics $(\sigma_i)_{i=0}^n$ and $(\tau_i)_{i=0}^n$ are close to each other ($2m \in \{0,2,n-2,n-1,n\}$), we proceed analogously. One observes that the auxiliary $\mathrm{CAT}(0)$ geodesic $\alpha'$ joining $x$ and $y$ passes at the $\mathrm{CAT}(0)$ distance at most $1$ from all the vertices of $(\delta_i)_{i=0}^n$ in all the appropriate layers. The rest of the argument goes the same as in the first case (i.e., one combines the above observation with \eqref{eq:flat2} and \eqref{eq:flat3} and obtains the same estimate as in \eqref{eq:duzelk}).\end{proof}

A \emph{flat} in a systolic complex $X$ is an isometric embedding $F \colon \E \hookrightarrow X$. 

\begin{lem}\label{lem:goodinflatisgood} Let $F \colon \E \hookrightarrow X$ be a flat, and suppose that $\gamma \subset \E$ is a $C'$--good geodesic (where $\E$ is treated as a systolic complex on its own). Then $F(\gamma)$ is a $(C'+10)$--good geodesic in $X$.
\end{lem}

\begin{proof}Pick any two vertices $x,y \in \E$, let $(\delta_i)_{i=0}^m \subset \E$ be the Euclidean geodesic between $x$ and $y$ in $\E$, and let $(\widetilde{\delta}_i)_{i=0}^m \subset X$ be the Euclidean geodesic between $F(x)$ and $F(y)$ in $X$. To prove the lemma it is enough to prove the following claim.

\begin{claim0}For any $i \in \{0, 1, \ldots, m\}$ for any two vertices $z_i \in \delta_i $ and $\widetilde{z}_i \in \widetilde{\delta}_i$ we have \[d(F(z_i), \widetilde{z}_i) \leqslant 10.\] 
\end{claim0}

The rest of the argument is devoted to proving the claim. Let $(\sigma_i)_{i=0}^n$ and $(\tau_i)_{i=0}^n$ be directed geodesics in $\E$ going respectively from $x$ to $y$ and from $y$ to $x$, and let $(\widetilde{\sigma}_i)_{i=0}^n$ and $(\widetilde{\tau}_i)_{i=0}^n$ be the corresponding directed geodesics between $F(x)$ and $F(y)$ in $X$ (see Convention~\ref{conv:directionconvention}). As in Lemma~\ref{lem:cat0inflatisgood}, we will first deal with the case when there is a single thick interval for $(\sigma_i)_{i=0}^n$ and $(\tau_i)_{i=0}^n$, i.e.,\ with the notation from Lemma~\ref{lem:cat0inflatisgood}, for $(\sigma_i)_{i=0}^n$ we assume that $2<2m <n-2$.

Consider the thick interval $(2, n-2)$ and for every $i \in \{2, 3, \ldots, n-2 \}$ choose vertices $s_i \in \sigma_i$ and $t_i \in \tau_i$ that realise the thickness of layer $i$ in $\E$. Let $S \colon \Delta \to \E$ be a characteristic surface for the cycle \[\alpha= (s_2, s_3, \ldots, s_{n-2}, t_{n-2}, t_{n-1}, \ldots, t_2, s_2).\] Observe that $S$ is the unique characteristic surface for the interval $(2, n-2)$, and that it is an isometric embedding. Therefore we will identify the characteristic disk $\Delta$ with its image $S(\Delta) \subset \E$.

The map $F\big|_{\Delta} \colon \Delta \to X$ does not have to be a characteristic surface for geodesics $(\widetilde{\sigma}_i)_{i=0}^n$ and $(\widetilde{\tau}_i)_{i=0}^n$, e.g., not all of the vertices of $F(\alpha)$ belong to the appropriate simplices of  $(\widetilde{\sigma_i})_{i=0}^n$ and $(\widetilde{\tau}_i)_{i=0}^n$. A priori we do not even know whether $(2,n-2)$ is a thick interval for these geodesics.

We will now show how to modify $F$ to obtain a characteristic surface for $(\widetilde{\sigma}_i)_{i=0}^n$ and $(\widetilde{\tau}_i)_{i=0}^n$. The idea is that the images $F((\sigma_i)_{i=0}^n)$ and $F((\tau_i)_{i=0}^n)$) are $1$--close to the geodesics $(\widetilde{\sigma_i})_{i=0}^n$ and $(\widetilde{\tau}_i)_{i=0}^n$ respectively, and therefore a small perturbation of the map $F$ would give the desired characteristic surface. We will now make this idea precise. For every $i \in \{2,4,6, \ldots, 2m\}$ choose any vertex $u_i \in \widetilde{\sigma}_i$ and for every $i \in \{n-2m, n-2m+2,n-2m+4, \ldots, n-2 \}$ choose any vertex $w_i \in \widetilde{\tau}_i$. Then by \cite[Lemma 3.9]{E3} (applied simultaneously to geodesics $(\sigma_i)_{i=0}^n$ and $(\tau_i)_{i=0}^n$) there exists a flat $F' \colon \E \to X$ such that:

\begin{enumerate}

\item for $i \in \{2,4, \ldots, 2m\}$ we have $F'(\sigma_i)=u_i$,
\item for $i \in \{n-2m, n-2m+2, \ldots, n-2 \}$ we have $F'(\tau_i)=w_i$,
\item for every vertex $x \in \E$ not considered in (1) and (2), we have $F'(x)=F(x)$,
\item for every $0 \leqslant i \leqslant n$ we have $F'(\sigma_i)= \mathrm{Im}(F') \cap \widetilde{\sigma}_i$ and $F'(\tau_i)= \mathrm{Im}(F') \cap \widetilde{\tau}_i$.
\end{enumerate}
Images of $F$ and $F'$ are shown in Figure~\ref{fig:twoflats}.

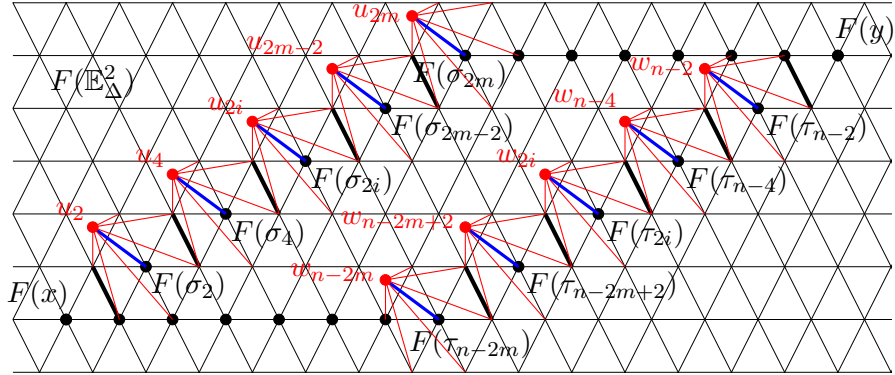
\begin{figure}[!h]
\centering
\begin{tikzpicture}[scale=0.7]
\begin{scope}
    \clip(-8,4) rectangle (8.5,-3);

\begin{scope}[color=black, very thin]
\draw (-8,0) to (10,0);
\draw (-8,1) to (10,1);
\draw (-8,2) to (10,2);
\draw (-8,3) to (10,3);
\draw (-8,4) to (10,4);
\draw (-8,5) to (10,5);
\draw (-8,6) to (10,6);
\draw (-8,-1) to (10,-1);
\draw (-8,-2) to (10,-2);
\draw (-8,-3) to (10,-3);
\draw (-8,-4) to (10,-4);
\draw(-8,4) to (-7,6);
\draw(-8,2) to (-6,6);
\draw(-8,0) to (-5,6);
\draw(-8,-2) to (-4,6);
\draw(-8,-4) to (-3,6);
\draw(-7,-4) to (-2,6);
\draw(-6,-4) to (-1,6);
\draw(-5,-4) to (0,6);
\draw(-4,-4) to (1,6);
\draw(-3,-4) to (2,6);
\draw(-2,-4) to (3,6);
\draw(-1,-4) to (4,6);
\draw(0,-4) to (5,6);
\draw(1,-4) to (6,6);
\draw(2,-4) to (7,6);
\draw(3,-4) to (8,6);
\draw(4,-4) to (9,6);
\draw(5,-4) to (10,6);
\draw(6,-4) to (10,4);
\draw(7,-4) to (10,2);
\draw(8,-4) to (10,0);
\draw(9,-4) to (10,-2);

\begin{scope}[yscale=-1,xscale=1, shift={(0,-2)}]
\draw(-8,4) to (-7,6);
\draw(-8,2) to (-6,6);
\draw(-8,0) to (-5,6);
\draw(-8,-2) to (-4,6);
\draw(-8,-4) to (-3,6);
\draw(-7,-4) to (-2,6);
\draw(-6,-4) to (-1,6);
\draw(-5,-4) to (0,6);
\draw(-4,-4) to (1,6);
\draw(-3,-4) to (2,6);
\draw(-2,-4) to (3,6);
\draw(-1,-4) to (4,6);
\draw(0,-4) to (5,6);
\draw(1,-4) to (6,6);
\draw(2,-4) to (7,6);
\draw(3,-4) to (8,6);
\draw(4,-4) to (9,6);
\draw(5,-4) to (10,6);
\draw(6,-4) to (10,4);
\draw(7,-4) to (10,2);
\draw(8,-4) to (10,0);
\draw(9,-4) to (10,-2);
\end{scope}
\end{scope}

\end{scope}

\draw[fill] (-7,-2)   circle [radius=0.1];
\draw[ultra thick] (-6, -2) --(-6.5, -1);

\begin{scope}[shift={(1.5,1)}]
\draw[fill] (-7,-2)   circle [radius=0.1];
\draw[ultra thick] (-6, -2) --(-6.5, -1);
\end{scope}
\begin{scope}[shift={(3,2)}]
\draw[fill] (-7,-2)   circle [radius=0.1];
\draw[ultra thick] (-6, -2) --(-6.5, -1);
\end{scope}
\begin{scope}[shift={(4.5,3)}]
\draw[fill] (-7,-2)   circle [radius=0.1];
\draw[ultra thick] (-6, -2) --(-6.5, -1);
\end{scope}
\begin{scope}[shift={(6,4)}]
\draw[fill] (-7,-2)   circle [radius=0.1];
\draw[ultra thick] (-6, -2) --(-6.5, -1);
\end{scope}

\draw[fill] (0.5,3)   circle [radius=0.1];
\draw[fill] (1.5,3)   circle [radius=0.1];
\draw[fill] (2.5,3)   circle [radius=0.1];
\draw[fill] (3.5,3)   circle [radius=0.1];
\draw[fill] (4.5,3)   circle [radius=0.1];
\draw[fill] (5.5,3)   circle [radius=0.1];
\draw[fill] (6.5,3)   circle [radius=0.1];
\draw[fill] (7.5,3)   circle [radius=0.1];


\node  [above left] at (-6.75,-2) {$F(x)$};
\node  [above right] at (7.25,3) {$F(y)$};

\node  [below right] at (-5.6,-0.9) {$F(\sigma_2)$};
\node  [below right] at (-5.6+1.5,-0.9+1) {$F(\sigma_4)$};
\node  [below] at (-5.6+6,-0.9+4) {$F(\sigma_{2m})$};
\node  [below right] at (-5.6+3,-0.9+2) {$F(\sigma_{2i})$};
\node  [below right] at (-5.6+4.5,-0.9+3) {$F(\sigma_{2m-2})$};

\begin{scope}[shift={(0.5,1)}, rotate=180]
\draw[fill] (-7,-2)   circle [radius=0.1];
\draw[ultra thick] (-6, -2) --(-6.5, -1);

\begin{scope}[shift={(1.5,1)}]
\draw[fill] (-7,-2)   circle [radius=0.1];
\draw[ultra thick] (-6, -2) --(-6.5, -1);
\end{scope}
\begin{scope}[shift={(3,2)}]
\draw[fill] (-7,-2)   circle [radius=0.1];
\draw[ultra thick] (-6, -2) --(-6.5, -1);
\end{scope}
\begin{scope}[shift={(4.5,3)}]
\draw[fill] (-7,-2)   circle [radius=0.1];
\draw[ultra thick] (-6, -2) --(-6.5, -1);
\end{scope}
\begin{scope}[shift={(6,4)}]
\draw[fill] (-7,-2)   circle [radius=0.1];
\draw[ultra thick] (-6, -2) --(-6.5, -1);
\end{scope}

\draw[fill] (0.5,3)   circle [radius=0.1];
\draw[fill] (1.5,3)   circle [radius=0.1];
\draw[fill] (2.5,3)   circle [radius=0.1];
\draw[fill] (3.5,3)   circle [radius=0.1];
\draw[fill] (4.5,3)   circle [radius=0.1];
\draw[fill] (5.5,3)   circle [radius=0.1];
\draw[fill] (6.5,3)   circle [radius=0.1];
\draw[fill] (7.5,3)   circle [radius=0.1];
\end{scope}

\node  [below right] at (-0.75,-2) {$F(\tau_{n-2m})$};
\node  [below right] at (1.45,-0.9) {$F(\tau_{n-2m+2})$};
\node  [below right] at (1.45+1.5,-0.9+1) {$F(\tau_{2i})$};
\node  [below right] at (1.45+4.5,-0.9+3) {$F(\tau_{n-2})$};
\node  [below right] at (1.45+3,-0.9+2) {$F(\tau_{n-4})$};


 

\draw[very thick, blue] (-6.5,-0.25)--(-5.5,-1);
\begin{scope}[red]
\draw[fill] (-6.5,-0.25) circle [radius=0.1];

\draw(-6.5,-0.25)-- (-6.5,-1);
\draw (-6.5,-0.25)-- (-4.5,-1);

\draw (-6.5,-0.25)-- (-5,0);
\draw  (-6.5,-0.25)-- (-6,0);
\draw  (-6.5,-0.25)-- (-5,-2);
\draw  (-6.5,-0.25)-- (-6,-2);
\node [above left] at (-6.45, -0.3) {$u_2$};
\end{scope}

\begin{scope}[shift={(1.5,1)}]
\draw[very thick, blue] (-6.5,-0.25)--(-5.5,-1);
\begin{scope}[red]
\draw[fill] (-6.5,-0.25) circle [radius=0.1];

\draw(-6.5,-0.25)-- (-6.5,-1);
\draw (-6.5,-0.25)-- (-4.5,-1);

\draw (-6.5,-0.25)-- (-5,0);
\draw  (-6.5,-0.25)-- (-6,0);
\draw  (-6.5,-0.25)-- (-5,-2);
\draw  (-6.5,-0.25)-- (-6,-2);
\node [above left] at (-6.45, -0.3) {$u_4$};
\end{scope}
\end{scope}

\begin{scope}[shift={(3,2)}]
\draw[very thick, blue] (-6.5,-0.25)--(-5.5,-1);
\begin{scope}[red]
\draw[fill] (-6.5,-0.25) circle [radius=0.1];

\draw(-6.5,-0.25)-- (-6.5,-1);
\draw (-6.5,-0.25)-- (-4.5,-1);

\draw (-6.5,-0.25)-- (-5,0);
\draw  (-6.5,-0.25)-- (-6,0);
\draw  (-6.5,-0.25)-- (-5,-2);
\draw  (-6.5,-0.25)-- (-6,-2);
\node [above left] at (-6.45, -0.3) {$u_{2i}$};
\end{scope}
\end{scope}

\begin{scope}[shift={(4.5,3)}]
\draw[very thick, blue] (-6.5,-0.25)--(-5.5,-1);
\begin{scope}[red]
\draw[fill] (-6.5,-0.25) circle [radius=0.1];

\draw(-6.5,-0.25)-- (-6.5,-1);
\draw (-6.5,-0.25)-- (-4.5,-1);

\draw (-6.5,-0.25)-- (-5,0);
\draw  (-6.5,-0.25)-- (-6,0);
\draw  (-6.5,-0.25)-- (-5,-2);
\draw  (-6.5,-0.25)-- (-6,-2);
\node [above left] at (-6.45, -0.2) {$u_{2m-2}$};
\end{scope}
\end{scope}

\begin{scope}[shift={(6,4)}]
\draw[very thick, blue] (-6.5,-0.25)--(-5.5,-1);
\begin{scope}[red]
\draw[fill] (-6.5,-0.25) circle [radius=0.1];

\draw(-6.5,-0.25)-- (-6.5,-1);
\draw (-6.5,-0.25)-- (-4.5,-1);

\draw (-6.5,-0.25)-- (-5,0);
\draw  (-6.5,-0.25)-- (-6,0);
\draw  (-6.5,-0.25)-- (-5,-2);
\draw  (-6.5,-0.25)-- (-6,-2);
\node [above left] at (-6.5, -0.55) {$u_{2m}$};
\end{scope}
\end{scope}


\begin{scope}[shift={(5.5,-1)}]

\draw[very thick, blue] (-6.5,-0.25)--(-5.5,-1);
\begin{scope}[red]
\draw[fill] (-6.5,-0.25) circle [radius=0.1];

\draw(-6.5,-0.25)-- (-6.5,-1);
\draw (-6.5,-0.25)-- (-4.5,-1);

\draw (-6.5,-0.25)-- (-5,0);
\draw  (-6.5,-0.25)-- (-6,0);
\draw  (-6.5,-0.25)-- (-5,-2);
\draw  (-6.5,-0.25)-- (-6,-2);
\node [above left] at (-6.5, -0.5) {$w_{n-2m}$};
\end{scope}

\begin{scope}[shift={(1.5,1)}]
\draw[very thick, blue] (-6.5,-0.25)--(-5.5,-1);
\begin{scope}[red]
\draw[fill] (-6.5,-0.25) circle [radius=0.1];

\draw(-6.5,-0.25)-- (-6.5,-1);
\draw (-6.5,-0.25)-- (-4.5,-1);

\draw (-6.5,-0.25)-- (-5,0);
\draw  (-6.5,-0.25)-- (-6,0);
\draw  (-6.5,-0.25)-- (-5,-2);
\draw  (-6.5,-0.25)-- (-6,-2);
\node [above left] at (-6.5, -0.5) {$w_{n-2m+2}$};
\end{scope}
\end{scope}

\begin{scope}[shift={(3,2)}]
\draw[very thick, blue] (-6.5,-0.25)--(-5.5,-1);
\begin{scope}[red]
\draw[fill] (-6.5,-0.25) circle [radius=0.1];

\draw(-6.5,-0.25)-- (-6.5,-1);
\draw (-6.5,-0.25)-- (-4.5,-1);

\draw (-6.5,-0.25)-- (-5,0);
\draw  (-6.5,-0.25)-- (-6,0);
\draw  (-6.5,-0.25)-- (-5,-2);
\draw  (-6.5,-0.25)-- (-6,-2);
\node [above left] at (-6.45, -0.3) {$w_{2i}$};
\end{scope}
\end{scope}

\begin{scope}[shift={(4.5,3)}]
\draw[very thick, blue] (-6.5,-0.25)--(-5.5,-1);
\begin{scope}[red]
\draw[fill] (-6.5,-0.25) circle [radius=0.1];

\draw(-6.5,-0.25)-- (-6.5,-1);
\draw (-6.5,-0.25)-- (-4.5,-1);

\draw (-6.5,-0.25)-- (-5,0);
\draw  (-6.5,-0.25)-- (-6,0);
\draw  (-6.5,-0.25)-- (-5,-2);
\draw  (-6.5,-0.25)-- (-6,-2);
\node [above left] at (-6.45, -0.2) {$w_{n-4}$};
\end{scope}
\end{scope}

\begin{scope}[shift={(6,4)}]
\draw[very thick, blue] (-6.5,-0.25)--(-5.5,-1);
\begin{scope}[red]
\draw[fill] (-6.5,-0.25) circle [radius=0.1];

\draw(-6.5,-0.25)-- (-6.5,-1);
\draw (-6.5,-0.25)-- (-4.5,-1);

\draw (-6.5,-0.25)-- (-5,0);
\draw  (-6.5,-0.25)-- (-6,0);
\draw  (-6.5,-0.25)-- (-5,-2);
\draw  (-6.5,-0.25)-- (-6,-2);
\node [above left] at (-6.5, -0.6) {$w_{n-2}$};
\end{scope}
\end{scope}

\end{scope}

\node [above left] at (-5.5, 2) {$F(\E)$};

\end{tikzpicture}
\caption{Images of flats $F$ and $F'$. The part of $F'(\E)$ which differs from $F(\E)$ is red.}
\label{fig:twoflats}
\end{figure}

Observe that for $i \in \{2,4, \ldots, 2m\}$ vertices $F(\sigma_i)$ and $F'(\sigma_i)=u_i$ are connected. Indeed, since $F$ and $F'$ agree on all the neighbours of $\sigma_i$ in $\E$, taking a pair of neighbours $v_1, v_2 \in \E$ of $\sigma_i$ that are not adjacent gives a $4$--cycle \[(F(v_1), F(\sigma_i), F(v_2), F'(\sigma_i)).\] Since $X$ is $6$--large, this cycle has a diagonal. It cannot be $[F(v_1), F(v_2)]$ since $F$ is an isometric embedding, and $v_1$ and $v_2$ are not adjacent. Thus it must be $[F(\sigma_i), F'(\sigma_i)]$.

Analogously, for any $i \in \{n-2m, n-2m+2, \ldots, n-2 \}$ vertices $ F(\tau_i)$ and $F'(\tau_i)=w_i$ are connected.
The above assertions, together with property (3) of the map $F'$, imply that for any vertex $x \in \E$ we have \[d(F(x), F'(x)) \leqslant 1.\]

We claim that for any $2 \leqslant i \leqslant n-2 $ the vertices $F'(s_i) \in \widetilde{\sigma}_i$ and $ F'(t_i) \in \widetilde{\tau}_i$ realise the thickness of layer $i$ for $(\widetilde{\sigma}_i)_{i=0}^n$ and $(\widetilde{\tau}_i)_{i=0}^n$ in $X$, and so layers $L_3,L_4, \ldots, L_{n-3}$ are thick, and layers $L_2$ and $L_{n-2}$ are thin. This follows essentially from the fact that $F' \colon \E \to X$ is an isometric embedding. Similarly, one can show that layers $L_1$ and $L_{n-1}$ in $X$ are thin. Finally, we conclude that \[F'\big|_{\Delta} \colon \Delta \to X\] is a characteristic surface for the thick interval $(2,n-2$) for geodesics $(\widetilde{\sigma}_i)_{i=0}^n$ and $(\widetilde{\tau}_i)_{i=0}^n$ in $X$. Let $(\rho_i)_{i=3}^{n-3}$ be the Euclidean diagonal for $\Delta \subset \E$ and 
denote by $(\delta_i)_{i=0}^m$ the Euclidean geodesic between vertices $x$ and $y$ in $\E$. We have $\delta_i=\mathrm{span}\{\sigma_i, \tau_i\}$ for $i \in \{1,2,n-2,n-1\}$. For all remaining $i$ we have $\delta_i=\rho_i$ (since $\Delta \cong S(\Delta) \subset \E$ is the unique characteristic surface for the thick interval $(2,n-2)$).
For any $i \in \{0, 1, \ldots, n\}$, for any vertex $z_i\in \delta_i$ we have 
\begin{equation}\label{eq:equiflats} 
d(F(z_i), F'(z_i)) \leqslant 1.
\end{equation}
For any $i \in \{0, 1, \ldots, n\}$, for any vertices $z_i \in \delta_i$ and $\widetilde{z}_i \in \widetilde{\delta}_i$ we claim that 
\begin{equation}\label{eq:charsurface} 
d(F'(z_i), \widetilde{z}_i) \leqslant 1.
\end{equation}

This follows for $i \in \{1,2,n-2,n-1\}$ from the property (4) of the map $F' \colon \E \to X$. Namely we have that $F'(\sigma_i)\subset \widetilde{\sigma}_i$ and $F'(\tau_i)\subset \widetilde{\tau}_i$, and by definition $\delta_i=\mathrm{span}\{\sigma_i, \tau_i\}$ and $\widetilde{\delta}_i=\mathrm{span}\{\widetilde{\sigma}_i, \widetilde{\tau}_i\}$.
For $i\in \{3, 4, \ldots, n-3 \}$ by definition of a Euclidean geodesic and the fact that $F'\big|_{\Delta}$ is a characteristic surface for $(\widetilde{\sigma}_i)_{i=0}^n$ and $(\widetilde{\tau}_i)_{i=0}^n$ we obtain that $F'(\rho_i) \subset \widetilde{\delta}_i$. 

Finally, combining \eqref{eq:equiflats} and \eqref{eq:charsurface}, for any $i \in \{0, 1, \ldots, n\}$ for any two vertices $z_i \in \delta_i$ and $\widetilde{z}_i \in \widetilde{\delta}_i$ we have
\begin{equation*}
d(F(z_i), \widetilde{z}_i) \leqslant 2.
\end{equation*}
This finishes the proof of the claim under the assumption that $2<2m <n-2$.\medskip

Now assume $2m \in \{0, 2, n-2, n-1,n\}$. In this case any layer $L_i \subset \E$ has thickness at most $2$. By \cite[Proposition 3.8]{E3} (which is a weaker formulation of \cite[Lemma 3.9]{E3} used above) for any vertices $s_i \in \sigma_i$ and  $u_i \in \widetilde{\sigma}_i$ we have $d(F(s_i), u_i) \leqslant 2$. The same estimate holds for vertices of $\tau_i$ and $\widetilde{\tau}_i$. It follows from the triangle inequality, that for any $i\in \{0, 1, \ldots, n\}$ the thickness of the layer $i$ in $X$ is at most $6$.

Observe that by definition of the Euclidean geodesic, any simplex $\delta_i$  lies between simplices $\sigma_i$ and $\tau_i$ in the layer $L_i \subset \E$. More precisely, the distance between any vertex $z_i \in \delta_i$ and any vertex $u_i \in \sigma_i$ is less than the thickness of $L_i$. Clearly the same estimate holds for vertices of $ \widetilde{\delta}_i$ and $\widetilde{\sigma}_i$, if one replaces thickness of $L_i$ by thickness of layer $i$ in $X$. From these considerations we conclude that for any $i\in \{0, 1, \ldots, n\}$, for any $z_i \in \delta_i$ and $\widetilde{z}_i \in \widetilde{\delta}_i$ and for any vertices $s_i \in \sigma_i$ and $u_i \in \widetilde{\sigma_i}$ we have:
\[d(F(z_i), \widetilde{z}_i) \leqslant d(F(z_i), F(s_i)) + d(F(s_i),u_i)+d(u_i, \widetilde{z}_i) \leqslant 2 +2+ 6  =10.\]
This estimate is by no means optimal.
\end{proof}

We are ready now to prove Theorem~\ref{tw:fixedpoints}.

\begin{proof}[Proof of Theorem~\ref{tw:fixedpoints}] 

\textbf{Case 1: $\minset{h}$ is $h$--cocompact.} Let $K$ be the constant appearing in Lemma~\ref{lem:euclideanisclosetomin}. Since $\minset{h}$ is $h$--cocompact, by Lemma~\ref{lem:disp} the subcomplex ${\disp{K}{h}}$ is $h$--cocompact as well. Pick a vertex $x \in \minset{h} \subset \disp{K}{h}$, and for any $n \geqslant 0$ consider vertices $h^{-n} \cdot x, h^{n} \cdot x \in \minset{h}$. Note that $d(h^{-n} \cdot x, h^{n} \cdot x)$ is not necessarily equal to $2n \cdot \trol{h}$, but we can assume that it is even (by passing to a subsequence of the form $n_i=ik$ for some $k \geqslant 1$ if necessary, see \cite[Theorem 1.1]{E2}). Put $m_n= \frac{1}{2} \cdot d(h^{-n} \cdot x, h^{n} \cdot x)$ and let $(\delta^n_i)_{i=-m_n}^{m_n}$ be the Euclidean geodesic between $h^{-n} \cdot x $ and $h^{n} \cdot x$.
By Lemma~\ref{lem:euclideanisclosetomin} we have $(\delta^n_i)_{i=-m_n}^{m_n} \subset \disp{K}{h}.$ Since $\disp{K}{h}$ is $h$--cocompact, there exists $R>0$ such that for every $n$ the geodesic $(\delta^n_i)_{i=-m_n}^{m_n}$ intersects the ball $B_R(x,X)$.  Let $i_n$ be an integer such that $\delta^n_{i_n}$ is a simplex of $(\delta^n_i)_{i=-m_n}^{m_n}$ that intersects $B_R(x,X)$ (such $i_n$ is not unique in general, we choose one for each $n$). 
By replacing $R$ with $R+1$ we can assume that $\delta^n_{i_n} \subset B_R(x,X)$.

Since the ball $B_R(x,X)$ contains only finitely many simplices, there are infinitely many $n$ such that $\delta^n_{i_n}$ is equal to a fixed simplex of $B_R(x,X)$. Denote this simplex by $\widetilde{\delta}_0$. Now since the sphere $S_1(\widetilde{\delta_0}, X)$ is finite, among geodesics $(\delta^n_i)_{i=-m_n}^{m_n}$ for which $\delta^n_{i_n} =\widetilde{\delta}_0$ there are infinitely many such that $\delta^n_{i_n+1}$ is equal to a fixed simplex $\widetilde{\delta}_{1}$ and $\delta^n_{i_n-1}$ is equal to a fixed simplex $\widetilde{\delta}_{-1}$. By continuing this procedure for spheres $S_k(\widetilde{\delta_0}, X)$ for $k>1$, we obtain a bi-infinite sequence of simplices \[(\widetilde{\delta}_i)_{i=-\infty}^{\infty} \subset \disp{K}{h},\] such that any of its finite subsequences is a Euclidean geodesic. 

By Proposition~\ref{prop:euclideanisgeo} for any finite subsequence, say $(\widetilde{\delta}_i)_{i=-m}^{m}$, there exists a geodesic $\gamma_m=(v_i)_{i=-m}^m$ such that $v_i\in \widetilde{\delta}_i$. By Theorem~\ref{tw:euclideanisgood} any $\gamma_m$ is a good geodesic. By a diagonal argument, from the sequence $(\gamma_m)_{m=0}^{\infty}$ we can extract a bi-infinite geodesic $\gamma=(v_i)_{i=-\infty}^{\infty}$, which is a good geodesic, as any of its finite subgeodesics is contained in a good geodesic $\gamma_m$ for some $m>0$. Since for every $i \in \mathbb{N}$ we have $v_i \in \widetilde{\delta}_i  \subset \disp{K}{h}$, we conclude that $\gamma \subset \disp{K}{h}$.\medskip 

\textbf{Case 2. $\minset{h}$ is not $h$--cocompact.} By \cite[Corollary 5.8]{OsaPry} the centraliser $C_G(h)$ is commensurable with the product $F_n \times \mathbb{Z}$, such that the subgroup $\langle h \rangle \subset C_G(h)$ is commensurable with the `$\mathbb{Z}$' factor of the latter. By Theorem~\ref{tw:centraliser} the group $C_G(h)$ acts cocompactly on $\minset{h}$. Since $\minset{h}$ is not $h$--cocompact, we conclude that $n \geqslant 1$, and so there exists an element $g \in C_G(h)$ such that $\langle g, h \rangle \cong \mathbb{Z}^2$. By the Flat Torus Theorem (\cite[Theorem 6.1]{E1}) there exists a flat $F \colon \E \to X$ whose image is preserved by the action of $\langle g, h \rangle$. We will now construct an $h$--invariant geodesic $\gamma \subset F(\E)$ which satisfies the assumptions of Lemma~\ref{lem:cat0inflatisgood}. 

Take any vertex $x \in F(\E)$ and consider a $\mathrm{CAT}(0)$ geodesic $\gamma'$ in $F(\E)$ that passes through vertices $x$ and $h \cdot x$. The isometry $h$ acts on $F(\E) \cong \mathbb{E}^2$ as a translation along $\gamma'$ by distance equal to the $\mathrm{CAT}(0)$ length of segment $\gamma'\big|_{[x, h \cdot x]}$. Let $\alpha$ be any (combinatorial) geodesic between $x$ and $h \cdot x$ that is Hausdorff $1$--close to $\gamma'\big|_{[x, h \cdot x]}$. (To obtain such $\alpha$ one proceeds similarly as when defining the Euclidean diagonal in a characteristic disk in Subsection~\ref{subsec:euclideandiag}.) Define $\gamma$ as \[\gamma= \bigcup_{n \in \mathbb{Z}} (h^n \cdot \alpha).\]

By definition $\gamma$ is an $h$--invariant geodesic, that is Hausdorff $1$--close to a $\mathrm{CAT}(0)$ geodesic $\gamma'$ in $F(\E)$. By Lemma~\ref{lem:cat0inflatisgood} we get that $\gamma$ is a $(\frac{4}{\sqrt{3}}+1)$--good geodesic in $F(\E)$. Lemma~\ref{lem:goodinflatisgood} implies that $\gamma$ is a $(\frac{4}{\sqrt{3}}+11)$--good geodesic in $X$. This implies that $\gamma$ is a good geodesic in $X$ since we have $\frac{4}{\sqrt{3}}+11<C$, where $C$ is the constant appearing in Definition~\ref{def:goodgeo} (cf.\ the discussion at the beginning of Subsection~\ref{subsec:goodgeo}).
\end{proof}

\end{document}